\documentclass[oneside, reqno,11pt]{amsart}
\usepackage{graphicx}
\usepackage{caption, subcaption, float}

\usepackage[margin=1in]{geometry}
\usepackage{amsmath}

\allowdisplaybreaks

\newcommand{\bea}{\begin{eqnarray}}
\newcommand{\eea}{\end{eqnarray}}
\newcommand{\beann}{\begin{eqnarray*}}
\newcommand{\eeann}{\end{eqnarray*}}

\newcommand{\n}{\noindent}

\usepackage{verbatim}
\usepackage{amssymb,amsmath,amsthm}
\usepackage[mathscr]{eucal}
\usepackage{tikz}
\usetikzlibrary{positioning}
\usepackage{pdfpages}
\usepackage{mathtools}
\usepackage{setspace}

\numberwithin{equation}{section}
\usepackage[toc,page]{appendix}

\usepackage{graphicx}
\usepackage{psfrag}
\usepackage{afterpage}
 \usepackage{array,multirow,graphicx}
\newtheorem{theorem}{Theorem}[section]
\newtheorem{conjecture}[theorem]{Conjecture}

\newtheorem{lemma}[theorem]{Lemma}

\newtheorem{rem}[theorem]{Remark}

\usepackage[colorlinks=true]{hyperref}
\usepackage{ifpdf}
\usepackage{float}

\title{Non-recursive Counts of Graphs on Surfaces}

\author
{Nicholas Ercolani$^1$}
\thanks{$^1$ University of Arizona, Department of Mathematics
  (ercolani@math.arizona.edu)}

\author{Joceline Lega$^2$}
\thanks{$^2$ University of Arizona, Department of Mathematics (lega@math.arizona.edu, \url{www.math.arizona.edu/\string~lega/}),}

\author{Brandon Tippings$^3$}
\thanks{$^3$ University of Arizona, Department of Mathematics ({tippings@arizona.edu}).}
 
\begin{document}
\maketitle

\begin{abstract}
   The problem of map enumeration concerns counting connected spatial graphs, with a specified number $j$ of vertices, that can be embedded in a compact surface of genus $g$ in such a way that its complement yields a cellular decomposition of the surface. As such this problem lies at the cross-roads of combinatorial studies in low dimensional topology and graph theory. The determination of explicit formulae for map counts, in terms of closed classical combinatorial functions of $g$ and $j$ as opposed to a recursive prescription, has been a long-standing problem with explicit results known only for very low values of $g$. In this paper we derive closed-form expressions for counts of maps with an arbitrary number of even-valent vertices, embedded in surfaces of arbitrary genus. In particular, we exhibit a number of higher genus examples for 4-valent maps that have not appeared prior in the literature.\\
   
   \n\textbf{Keywords:} map enumeration, analytical combinatorics, hypergeometric functions, transfer matrices\\
   
   \n\textbf{Mathematics Subject Classifications:} 05A15, 05A10, 33C05, 33C20, 39A60, 57K20
\end{abstract}


\section{Introduction}

This article concerns {\it map enumeration}, a topic that brings together the classical subjects of graphical enumeration and low-dimensional combinatorial topology. This subject has a long history going back to the work of Tutte and his school in the 60's. We refer to \cite{bib:tu} for a contemporaneous survey of that early work. Tutte's original interest concerned {\it rooted planar maps}, which correspond to unlabeled connected graphs embedded in the sphere (or equivalently the plane) that are rooted. In this context, rooted means that a vertex of the graph, together with an edge adjacent to it and a side of that edge, are distinguished. The definition of maps for higher genus is more involved (see following paragraph) and enumerations in these cases are much more challenging to achieve. Some early results for low genus surfaces using an extension of the idea of rooted maps were carried out by Brown \cite{bib:brown} and Arques \cite{bib:arques}. Tutte's recursive approach was extended by Bender and Canfeld \cite{bib:bc86}, who initiated the study of asymptotic (in vertex number) enumerations for general surfaces. A refined version of their recursion was subsequently obtained by Eynard \cite{bib:ebook}. Over the years, a number of related but alternative approaches to the enumeration problem have been developed. These too are recursive and combinatorial in nature, as described in \cite{bib:cms}, \cite{bib:bg12} and the references therein. The present study establishes non-recursive map counts, through an approach grounded in analytical statistical physics, as described in the remainder of this introduction. Our results have a number of interesting current and potential applications in their own right. In particular, the asymptotic enumerations derived in \cite{bib:bc86} can be directly obtained from the exact closed form expressions \eqref{eq:zg_parfrac} and \eqref{eq:eg_parfrac} used in this paper (see section \ref{sec:mc2-7}).

We define a  {\it map} to be a connected, labeled graph embedded injectively into a compact, oriented and connected topological surface so that the complement of the graph in the surface is a disjoint union of cells (sets homeomorphic to open discs). {\it Labeling} means that vertices as well as the half-edges (referred to as darts) are labeled counter-clockwise around each vertex. This sets up a correspondence between {\it matchings}, pairwise, of dart arrangements and equivalence classes of maps. Due to surface symmetries, this correspondence is not quite 1:1. Indeed, two labeled maps are said to be equivalent if there is an orientation-preserving homeomorphism of the surface to itself that leaves the graph unchanged. For fixed graph size (i.e. for a fixed number of vertices), the number of non-equivalent dart matchings (for a surface of fixed genus) is finite. This number is what we refer to as a {\it map count}. Quotienting these counts by the cardinality of all possible labelings, which for a map of size $j$ is $j!\, (2\nu)^j$, would then seem to yield a more geometric count of maps under a coarser equivalence relation that does not involve labelings. However due to the possibility of a surface having symmetries, our map counts will not in general be divisible by the afore-mentioned labeling cardinality.
 This is in complete analogy with descriptions of moduli spaces of conformal structures on Riemann surfaces \cite{bib:nash}, where symmetries lead to orbifold singularities on those spaces. We refer to \cite{bib:em, bib:ew} and references therein for further details.
The top panel of Figure \ref{fig:torus} provides an example of the correspondence between dart matchings and maps. This figure also illustrates the connection between surface topology and graph theory that map enumeration provides: one sees that the dual map associated to the graph, viewed as a 4-valent dart matching, yields a surface tessellation, or {\it tiling}, by topological quadrilaterals.

\begin{figure}
\centering
\includegraphics[width=.7 \linewidth]{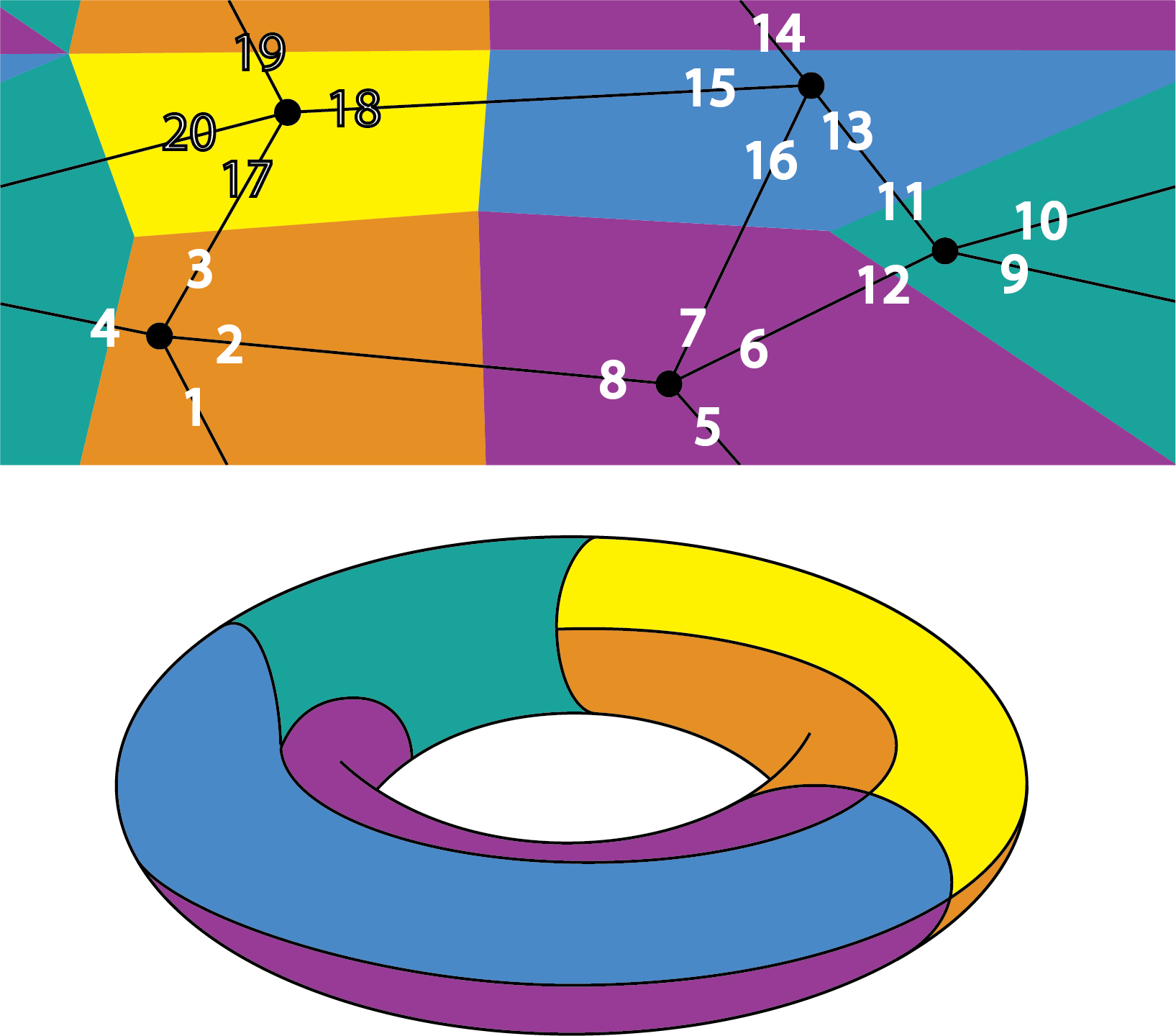}
\caption{Top: 4-valent map with 5 vertices on a torus. The numbers are dart labels and the colors represent the quadrilateral tiling induced by the dual graph. Bottom: Illustration of the resulting tiling of the torus with quadrilaterals.}
\label{fig:torus}
\end{figure}

In this paper we will restrict attention to two related cases: (a) maps whose vertices are all of degree $2 \nu$ (the regular $2 \nu$-valent case) whose count will be denoted by ${\mathcal N}_{2\nu,e}(g,j)$, and
(b) maps with $j$ vertices of degree $2 \nu$, and two additional vertices of degree 1, with count denoted by 
${\mathcal N}_{2\nu,z}(g,j)$. 
By definition the additional two vertices in the latter case each have a unique adjacent edge referred to as a {\it leg}. For case (b) the symmetries alluded to earlier are absent \cite{bib:elt22b} and so the map counts are each divisible by  $j!\, (2\nu)^j$. Hence, in this case the quotient counts the number of coarser map equivalence classes without labeling.

As previously mentioned, the method we employ to calculate map counts stems from statistical physics and more precisely from application of the Wick calculus to perturbation expansions of the random matrix partition function (for hermitian ensembles) and its correlation functions (referred to as genus expansions). This gives rise to a hierarchy of equations that recover  generating functions for the map counts as coefficients in the genus expansions \cite{bib:em}. These functions have the general form
\begin{equation*}
    \sum_{j = 0}^\infty \dfrac{{\mathcal N}(g,j)}{j!} t^j .
\end{equation*}
In what follows we denote the generating functions for counts of maps of genus $g$ ($g$-maps) in case (a) by $e_g$, and in case (b) by $z_g$.
The two are intimately related. In particular, the $e_g$ may be determined from the $z_g$ by solving a forced Cauchy-Euler equation \cite{bib:emp08}. 

There are two approaches in the literature for deriving recursion equations for the map generating functions from the genus expansions. One of these is based on the matrix resolvent (Green's function) for random matrices \cite{bib:a, bib:ebook}. This is often referred to as the {\it topological recursion}. The other is based on the recurrence matrix resolvent of orthogonal polynomials associated to the probability measure of the random matrix ensemble \cite{bib:em, bib:emp08, bib:bd13, bib:ew}. In the latter approach one is able to bring to bear rigorous analytical methods from the classical theory of orthogonal polynomials as well as Riemann-Hilbert analysis in order to get detailed estimates on the validity of the genus expansion and associated recurrence equations. This enables one to establish that the generating functions of the form displayed above, initially taken to be formal, are in fact convergent \cite{bib:em} and to determine their maximal domains of holomorphy. This is essential for the derivation of explicit closed rational expressions for the $z_g$ initiated in \cite{bib:er} and continued in \cite{bib:elt22b} that we use. We refer the reader to the last citation for additional details and will take as our starting point the rational expressions for $z_g$ and $e_g$ that were developed in \cite{bib:emp08, bib:er, bib:er14}.

While map enumeration has stemmed from recursion schemes of the type just mentioned, what we do in this paper is decidedly different. We aim to derive expressions for the map counts ${\mathcal N}(g,j)$ that are of {\em closed form} in terms of {\em known} classical combinatorial functions. This is what we refer to as a {\it non-recursive count}. A few examples are available in the literature: ${\mathcal N}_{2\nu,z}(0,j)$ and ${\mathcal N}_{2\nu,e}(0,j)$ were derived in \cite{bib:emp08}, ${\mathcal N}_{4,e}(2,j)$ was introduced in \cite{bib:biz}, and more recently, ${\mathcal N}_{4,e}(3,j)$ was obtained in \cite{bib:bgm}. To the best of our knowledge, the only other instances of known closed-form expressions are ${\mathcal N}_{3,e}(0,j)$ and ${\mathcal N}_{3,e}(1,j)$, for triangulations on a surface of genus 0 or 1  \cite{bib:bd13}. This article establishes general expressions for ${\mathcal N}_{2\nu,z}(g,j)$ and ${\mathcal N}_{2\nu,e}(g,j)$ in terms of a linear combination of a specified set of coefficients appearing in the rational forms for $z_g$ and $e_g$, weighted by hypergeometric functions depending on the genus $g$ and the number of vertices $j$. In addition, we introduce a dynamic perspective which, in the case of 4-valent maps, leads to a remarkable analytical simplification that enables us to efficiently express the desired non-recursive counts in terms of polynomials, powers, and factorials of the number of vertices $j$, for fixed but arbitrary genus, as shown in Tables \ref{tab:counts_z} and \ref{tab:counts_e} below. These tables suggest a general structure for the counts ${\mathcal N}_{4,z}(g,j)$ and ${\mathcal N}_{4,e}(g,j)$ of the form
\[
{\mathcal N}_{4,z}(g,j) = (-1)^g\, 12^j \left(\prod_{k=0}^{2g-2}(j-k)\right)
  \left(\frac{\left(2 j \right)!}{j !}\, P^{(g)}_{\lfloor g/2 \rfloor}(j) - \big(4^{j} j !\big)\, P^{(g)}_{\lfloor (g-1)/2 \rfloor}(j)\right), \quad g \ge 1, \ j \ge 0
\]
and
\[
{\mathcal N}_{4,e}(g,j) = (-1)^g\, 12^{j-1} \left(\prod_{k=1}^{2g-3}(j-k)\right)
  \left(\frac{\left(2 j \right)!}{j !}\, Q^{(g)}_{\lfloor g/2 \rfloor}(j) - \big(4^{j-1} j !\big)\, Q^{(g)}_{\lfloor (g-1)/2 \rfloor}(j)\right), \quad g \ge 2, \ j \ge 1
\]
where $P^{(g)}_h(j)$ and $Q^{(g)}_h(j)$ denote polynomials of degree $h$ in $j$ whose coefficients are positive rational numbers that depend only on $g$. The reduction leading to these expressions works for all genera in the case of ${\mathcal N}_{4,z}(g,j)$, and for $g \ge 2$ in the case of ${\mathcal N}_{4,e}(g,j)$. Counts of regular genus 1 maps are indeed special, and we provide a separate derivation for ${\mathcal N}_{4,e}(1,j)$ in Appendix \ref{app:e1}. 

\begin{center}
\begin{table}[h]
\def\arraystretch{2.4}%
\begin{tabular}{|c||c|}
\hline
  genus $g$ &  ${\mathcal N}_{4,z}(g,j)$ \\
  \hline
  1   &  $\displaystyle 12^{j} j \left(\frac{4^{j} j !}{12}-\frac{\left(2 j \right)!}{6 j !}\right)$\\
  \hline
  2 & $\displaystyle 12^{j} \left(\prod_{k=0}^2 (j-k)\right)
  \left(\frac{\left(2 j \right)!}{j !}\frac{7 (2 j +3)}{1080}-4^{j} j !\,\frac{7}{384}\right)$\\
  \hline
  3 & $\displaystyle 12^{j} 
  \left(\prod_{k=0}^4(j-k)\right)
  \left(4^{j} j ! \left(\frac{245 j}{497664}+\frac{12041}{4976640}\right)-\frac{\left(2 j \right)!}{j !}\frac{484 j +279}{136080}\right)$\\
  \hline
  4 & $\begin{matrix}\displaystyle 12^{j} \left(\prod_{k=0}^6 (j-k)\right)& \displaystyle
  \hskip -.3truecm \left(\frac{\left(2 j \right)!}{j !}\left(\frac{37079}{750578400} j^{2}+\frac{6067121}{10508097600} j +\frac{127}{604800}\right)\right.\\
  & \displaystyle \left. -4^{j} j ! \left(\frac{7805 j}{47775744}+\frac{1699447}{6688604160}\right)\right)
\end{matrix}$\\
  \hline
  5 & $\begin{matrix} \displaystyle 12^{j} \left(\prod_{k=0}^8 (j-k)\right)&
  \displaystyle \hskip -.3truecm \left(4^{j} j ! \left(\frac{38213}{27518828544} j^{2}+\frac{1702225}{55037657088} j +\frac{482999}{20384317440}\right)\right. \\
  & \displaystyle \left. -\frac{\left(2 j \right)!}{j !}\left(\frac{491951}{25519665600} j^{2}+\frac{1849339}{25519665600} j +\frac{73}{3421440}\right)\right)
 \end{matrix}$ \\
  \hline
  6 & $\begin{matrix}\displaystyle 12^{j} \left(\prod_{k=0}^{10} (j-k)\right)&
  \displaystyle \hskip -.3truecm \left(\frac{\left(2 j \right)!}{j !}\left(\frac{5004682489}{45165980162160000} j^{3}+\frac{389578665043}{92213876164410000} j^{2}\right. \right.\\
  & \displaystyle \left. \hskip .8truecm  +\frac{69512878587263}{8852532111783360000} j +\frac{1414477}{653837184000}\right) \\ 
  & \displaystyle \hskip -1.3truecm -4^{j} j ! \left(\frac{54362497}{87179648827392} j^{2}+\frac{381046393}{87179648827392} j \right.\\
  & \displaystyle \left. \left. \hskip 3.9truecm +\frac{43567716553}{20341918059724800}\right)\right)
 \end{matrix}$\\
  \hline
  7 & $\begin{matrix} \displaystyle 12^{j} \left(\prod_{k=0}^{12} (j-k)\right)&
  \displaystyle \hskip -1.truecm \left(4^{j} j ! \left(\frac{6334396069}{2448004539073167360} j^{3}+\frac{2801562779}{18133366956097536} j^{2}\right. \right.\\
  & \displaystyle \left. \hskip .8truecm +\frac{5032281513503}{9792018156292669440} j +\frac{46115735865131}{228480423646828953600}\right)\\
  & \displaystyle \hskip -.5truecm -\frac{\left(2 j \right)!}{j !}\left(\frac{953637649}{16937242560810000} j^{3}+\frac{335779266491}{491807339543520000} j^{2} \right.\\
  & \displaystyle \left. \left. \hskip .5truecm +\frac{20962080883129}{26557596335350080000} j
  +\frac{8191}{37362124800}\right)\right)
\end{matrix}$\\
  \hline
\end{tabular}
\caption{Expressions for the number of 2-legged 4-valent $g$-maps with $j$ vertices.}
\label{tab:counts_z}
\end{table}
\end{center}

\begin{center}
\begin{table}[h]
\bgroup
\def\arraystretch{2.5}%
\begin{tabular}{|c||c|}
\hline
  genus $g$ &  ${\mathcal N}_{4,e}(g,j)$ \\
  \hline
  1 & $\displaystyle j!\, 12^{j-1}
\left({2 j-1 \choose j-1}\ {_3}F_2 
\left( \genfrac{}{}{0pt}{}
{1,1,1-j}{2,j+1}; -1 \right)
 - {2 j-1 \choose j-2}\ {_3}F_2 
\left( \genfrac{}{}{0pt}{}
{1,1,2-j}{2,j+2}; -1 \right)\right)$\\
\hline
  2$^{(*)}$ & $\displaystyle 12^{j -1} \left(j -1\right) \left(\frac{ \left(2 j \right)!}{j !}\left(\frac{7 j}{90}+\frac{1}{40}\right)-4^{j -1} j !\,\frac{13}{48}\right)$\\
  \hline
  3$^{(*)}$ & $\displaystyle 12^{j -1}
  \left(\prod_{k=1}^3(j-k)\right)
  \left(4^{j -1} j ! \left(\frac{245 j}{20736}+\frac{781}{41472}\right)-\frac{\left(2 j \right)!}{j !}\left(\frac{337 j}{22680}+\frac{1}{1008}\right) \right)$\\
  \hline
  4 & $\begin{matrix}\displaystyle 12^{j -1}
  \left(\prod_{k=1}^5(j-k)\right)& \displaystyle \hskip -.3truecm
  \left(\frac{\left(2 j \right)!}{j !}\left(\frac{37079}{125096400} j^{2}+\frac{86356}{54729675} j +\frac{1}{28800}\right)\right.\\
  & \left. \hskip -1.truecm \displaystyle \hskip -.6 truecm - 4^{j -1} j ! \left(\frac{5845 j}{1990656}+\frac{23297}{39813120}\right)\right)
\end{matrix}$\\
  \hline
  5 & $\begin{matrix} \displaystyle 12^{j -1} \left(\prod_{k=1}^7(j-k)\right)&
  \displaystyle \hskip -.3truecm \left(4^{j -1} j ! \left(\frac{38213}{1146617856} j^{2}+\frac{915313}{2293235712} j -\frac{1940327}{53508833280}\right) \right.\\
  & \left. \displaystyle \hskip -.5truecm -\frac{\left(2 j \right)!}{j !} \left(\frac{211033}{2319969600} j^{2}+\frac{8139013}{71455063680} j +\frac{1}{887040}\right)\right)
 \end{matrix}$ \\
  \hline
  6 & $\begin{matrix}\displaystyle 12^{j -1} \left(\prod_{k=1}^9(j-k)\right) & \displaystyle 
 \hskip -.3truecm \left(\frac{\left(2 j \right)!}{j !}\left(\frac{5004682489}{7527663360360000} j^{3}+\frac{7523688218141}{491807339543520000} j^{2}\right.\right.\\
  & \displaystyle \left. +\frac{20903746897}{3944978659440000} j +\frac{691}{19813248000}\right)\\
  & \displaystyle \hskip -1.2truecm -4^{j -1} j !  \left(\frac{44274265}{3632485367808} j^{2}+\frac{135152437}{3632485367808} j \right.\\
  & \displaystyle \left. \left. -\frac{522404797}{77052719923200}\right)\right)
\end{matrix}$\\
  \hline
  7 & $\begin{matrix} \displaystyle
  12^{j -1} \left(\prod_{k=1}^{11}(j-k)\right) & \displaystyle 
  \hskip -.5truecm \left(4^{j -1} j ! \left(\frac{6334396069}{102000189128048640} j^{3}+\frac{27364604401}{11333354347560960} j^{2}\right.\right.\\
  & \displaystyle \left. \hskip 1.3truecm
  +\frac{988175350991}{408000756512194560} j -\frac{358193577649}{732309050150092800}\right)\\
  & \displaystyle \hskip -.1truecm -\frac{\left(2 j \right)!}{j !}\left(\frac{25511722279}{90331960324320000} j^{3}+\frac{2675917530049}{1475422018630560000} j^{2}\right.\\
  & \displaystyle \left. \left. \hskip -.3truecm
  +\frac{5035943441}{69980491002240000} j + \frac{1}{958003200}\right)\right)
\end{matrix}$\\
  \hline
\end{tabular}
\egroup
\caption{Expressions for the number of 4-valent $g$-maps with $j$ vertices ($j \ge 1$). In the first row, ${_3}F_2$ is the generalized hypergeometric function. Starred $^{(*)}$ rows correspond to results previously known in the literature: \cite{bib:biz} for $g=2$ and \cite{bib:bgm} for $g=3$.} \label{tab:counts_e}
\end{table}
\end{center}

The remainder of this article is organized as follows. In Section \ref{sec:2}, we derive closed-form expressions for the map counts in terms of contour integrals. In Section \ref{sec:rec_relations}, we develop vector difference equations that will lead to a different formulation of the map counts in the 4-valent case. Section \ref{sec:4-valent} shows that the dynamic description associated with the vector difference equations simplifies dramatically when $\nu = 2$. This allows us to formulate map counts in terms of the contraction of two vectors: a universal row vector truncated to a finite, genus-dependent length, and a column vector of the same size, whose entries are specific initial coefficients. It also explains how the expressions in Tables \ref{tab:counts_z} and \ref{tab:counts_e} are obtained. The conclusions, in Section \ref{sec:7}, point to some future directions of study based on the results of this paper. Finally, Appendices \ref{app:e1} and \ref{app:D} provide detailed calculations of the counts in terms of contour integrals, Appendices \ref{app:4-valent_dynamics} and \ref{app:map_cts_dot_product} describe the 4-valent case, and Appendices \ref{app:zg} and \ref{app:eg} respectively list the initial vectors of coefficients used to establish the formulas of Tables \ref{tab:counts_z} and \ref{tab:counts_e}.

\section{Generating Functions and closed-form counts} \label{sec:2}
It was shown in \cite{bib:emp08} that the number of $2\nu$-regular maps with $j$ vertices and two legs that can be embedded in a surface of genus $g$ may be expressed as the following derivative
\begin{equation}
\label{eq:map_counts_z}
{\mathcal N}_{2\nu,z}(g,j) = \left. (-1)^j \, \dfrac{d^j z_g}{d t^j}\right\vert_{t=0}.
\end{equation}
Here, $z_g$ is a rational functions of $z_0$, 
\[
z_g(z_0) = \dfrac{z_0 (z_0-1) P_{3 g - 2}(z_0)}{\left(\nu - (\nu - 1) z_0\right)^{5 g - 1}},
\]
where $P_{3 g - 2}(z_0)$ is a polynomial of degree $3 g - 2$ in $z_0$, and $z_0$ is a function of $t$ such that 
\begin{equation}
\label{eq:string}
1 = z_0 + 2 \nu \, t {2 \nu - 1 \choose \nu - 1} z_0^\nu.
\end{equation}
Equation \eqref{eq:string} is known as the string equation. Note that $z_0 = 1$ when $t=0$ and that implicit differentiation with respect to $t$ leads to 
\begin{equation}
\label{eq:dz_dt}
\dfrac{d z_0}{d t} = - c_\nu \, \dfrac{z_0^{\nu+1}}{\nu - (\nu - 1) z_0}, \qquad c_\nu = 2 \nu {2 \nu - 1 \choose \nu - 1}.
\end{equation}
The derivative in Equation \eqref{eq:map_counts_z} can thus be conveniently calculated by finding its expression in terms of $z_0$ only, before setting $z_0=1$. In addition, writing $z_g/z_0$ as a partial fraction expansion in powers of $\nu - (\nu-1) z_0$, we have
\begin{equation} \label{eq:zg_parfrac}
z_g (z_0) = z_0 \sum_{\ell = 0}^{3 g-1} \dfrac{a_{z,\ell}^{(0)}(g,\nu)}{\left(\nu - (\nu-1) z_0\right)^{2g+\ell}}, \qquad a_{z,\ell}^{(0)}(g,\nu) \in \mathbb{R}.
\end{equation}
Similarly, the number of $2\nu$-regular maps with $j$ vertices that can be embedded in a surface of genus $g$ is given by \cite{bib:emp08,bib:er14}
\begin{equation}
\label{eq:map_counts_e}
{\mathcal N}_{2\nu,e}(g,j) = \left. (-1)^j \, \dfrac{d^j e_g}{d t^j}\right\vert_{t=0},
\end{equation}
where, for $g \ge 2$,
\[
e_g(z_0)= \dfrac{(z_0-1)^r
Q_{5g-5-r}(z_0)}{\left(\nu - (\nu-1) z_0\right)^{5g-5}}, \qquad r = \max\left\{1, \left\lfloor \dfrac{2 g - 1}{\nu - 1} \right\rfloor \right\},
\]
and $Q_{5g-5-r}(z_0)$ is a polynomial of degree $5 g - 5 - r$ in $z_0$. We write the partial fraction expansion of $e_g$ in powers of $\nu-(\nu-1)z_0$ as \cite{bib:er14}
\begin{equation}
\label{eq:eg_parfrac}
e_g(z_0)=C^{(g)}+\sum_{\ell=0}^{3g-3}\dfrac{b_{e,\ell}^{(0)}(g,\nu)}{(\nu-(\nu-1) z_0)^{2g+\ell-2}}.
\end{equation}
Below, when it is clear from context, we suppress the $g$ and $\nu$ dependence of $a_{z,\ell}^{(0)}(g,\nu)$ and $b_{e,\ell}^{(0)}(g,\nu)$.

As defined in \eqref{eq:map_counts_z} and \eqref{eq:map_counts_e}, ${\mathcal N}_{2\nu,z}(g,j)$ and ${\mathcal N}_{2\nu,e}(g,j)$ are obtained by taking successive derivatives of the generating functions evaluated at $t = 0$. But from the convergence of their series expansions \cite{bib:em}, it follows that these functions are analytic, at least in the vicinity of $t=0$. Consequently, the counts may also be expressed as contour integrals,
\begin{eqnarray}
{\mathcal N}_{2\nu,z}(g,j) &=& 
\dfrac{j!\, c_\nu^j}{2\pi i}\oint \dfrac{z_g(\eta)}{\eta^{j+1}} d\eta \label{eq:z_counts_Cauchy}\\
{\mathcal N}_{2\nu,e}(g,j) &=&
\dfrac{j!\, c_\nu^j}{2\pi i}\oint \dfrac{e_g(\eta)}{\eta^{j+1}} d\eta, \label{eq:e_counts_Cauchy}
\end{eqnarray}
where $\eta = - c_\nu t$ and the variable $z_0$ are related by the string equation $1 = z_0 - \eta z_0^\nu$, and the integral is taken along a sufficiently small contour encircling the origin. In Appendices \ref{app:e1} and \ref{app:D},
we show that a direct evaluation of the above contour integrals leads to expressions for ${\mathcal N}_{2\nu,z}(g,j)$ and ${\mathcal N}_{2\nu,e}(g,j)$ in terms of hypergeometric functions, as stated in the following theorem. 
\begin{theorem} \label{th:counts_hypergeom}
The number of $2\nu$-valent maps with $j$ vertices and two legs that can be embedded in a surface of genus $g \ge 1$ is given by
\begin{equation}
\label{eq:z_couts_hypergeom}
{\mathcal N}_{2\nu,z}(g,j) = j!\, c_\nu^j (\nu-1)^{j} \sum_{\ell = 0}^{3g-1} \left(a_{z,\ell}^{(0)} {(2g-2) + (\ell+j)  \choose j} 
\ {_2}F_1
\left( \genfrac{}{}{0pt}{}
{-j,\ -\nu j}{2-2g-(\ell+j)}; \frac{1}{1-\nu} \right) \right)
\end{equation}
and the number of $2\nu$-valent $g$-maps with $j \ge 1$ vertices is, for $g \ge 2$,
\begin{equation}
\label{eq:e_couts_hypergeom}
{\mathcal N}_{2\nu,e}(g,j) = j!\, c_\nu^j (\nu - 1)^j  \sum_{\ell = 0}^{3g-3} \left(b_{e,\ell}^{(0)}
{(2g-4) + (\ell+j)  \choose j}
\ {_2}F_1
\left( \genfrac{}{}{0pt}{}
{-j,\ 1- \nu j}{4-2g-(\ell+j)}; \frac{1}{1-\nu} \right)\right),
\end{equation}
where ${_2}F_1$ denotes the Gauss hypergeometric function \cite{bib:nist}. The coefficients $a_{z,\ell}^{(0)}$ and $b_{e,\ell}^{(0)}$ depend on $g$ and $\nu$ and are defined in Equations \eqref{eq:zg_parfrac} and \eqref{eq:eg_parfrac} respectively. For regular $1$-maps and $j \ge 1$,
\begin{eqnarray}
\label{eq:e1_counts}
{\mathcal N}_{2\nu,e}(1,j) &=&
\dfrac{j!\, c_\nu^j}{12}
\left( (\nu - 1) {\nu j-1 \choose j-1}\, {_3}F_2 \left( \genfrac{}{}{0pt}{}
{1,\ 1,\ 1-j}{2,(\nu-1)j+1}; 1-\nu \right) \right. \nonumber \\
&& \qquad \quad \left. - (\nu - 1)^2 {\nu j-1 \choose j-2}\ {_3}F_2
\left( \genfrac{}{}{0pt}{}
{1,\ 1,\ 2-j}{2,(\nu-1)j+2}; 1-\nu \right)\right),
\end{eqnarray}
where ${_3}F_2$ is the generalized hypergeometric function \cite{bib:nist}.
\end{theorem}
In Theorem \ref{th:counts_hypergeom}, the hypergeometric functions with negative parameters are defined through the convention
\[
{_m}F_n \left( \genfrac{}{}{0pt}{}
{-j,\ k_1,\, k_2,\, \dots k_{m-1}}{-j-h,\ h_1,\, \cdots,h_{n-1}}; x \right) = \lim_{b \to -j-h} \left(\lim_{a \to -j}\ {_m}F_n \left( \genfrac{}{}{0pt}{}
{a,\ k_1,\, k_2,\, \dots k_{m-1}}{b,\ h_1,\,\cdots,h_{n-1}}; x \right) \right), \quad h \ge 0.
\]
This is in line with Equation 15.2.5 of  \cite{bib:nist} and consistent with the definition used by Maple \cite{bib:maple} when evaluating these functions. Consequently, the expressions for ${\mathcal N}_{2\nu,z}(g,j)$, $g \ge 1$, and ${\mathcal N}_{2\nu,e}(g,j)$, $g \ge 2$, are linear combinations of the $a_{z,\ell}^{(0)}$ and $b_{e,\ell}^{(0)}$, with finite hypergeometric coefficients. A general framework to obtain all of the numbers 
\[
\left\{a_{z,\ell}^{(0)},\, 0 \le \ell \le 3 g - 1 \right\} \qquad \text{and} \qquad \left\{b_{e,\ell}^{(0)},\, 0 \le \ell \le 3 g - 3 \right\},
\]
was developed in \cite{bib:elt22b}, which also provided expressions for $z_g$ and $e_g$ for $g \le 7$ in the 4-valent ($\nu=2$ case). Together with the results of \cite{bib:elt22b}, Theorem \ref{th:counts_hypergeom} therefore provides closed-form expressions of even-valent map counts on surfaces of arbitrary genus.

In the next section, we build on the analysis of \cite{bib:er, bib:er14}, which was motivated by the structure of the Cauchy integral representations of Equations \eqref{eq:z_counts_Cauchy} and \eqref{eq:e_counts_Cauchy}, to describe, in terms of vector difference equations, the dynamics of the coefficients $a_{z,\ell}^{(j)}$ and $b_{e,\ell}^{(j)}$ appearing in the partial fraction expansions of the $j^{th}$ derivatives of $z_g$ and $e_g$. This will allow us to derive the expressions (other than ${\mathcal N}_{4,e}(1,j)$) provided in Tables \ref{tab:counts_z} and \ref{tab:counts_e} in the 4-valent case.

\section{Recurrence relations} \label{sec:rec_relations}
Because of the presence of hypergeometric functions, the closed-form expressions given in Theorem \ref{th:counts_hypergeom} involve sums whose number of terms grows with $j$. In contrast, the formulas given in Table \ref{tab:counts_z} for $g \ge 1$ and Table \ref{tab:counts_e} for $g \ge 2$ do not. Such a remarkable property is the consequence of a simplification that occurs in the 4-valent case, the nature of which may be understood by following the dynamics of the coefficients $a_{z,\ell}^{(j)}$ and $b_{e,\ell}^{(j)}$. This and the next sections are devoted to analyzing this phenomenon.

Taking the derivative of \eqref{eq:zg_parfrac} with respect to $t$ and using the expression of $d z_0 / d t$ given in \eqref{eq:dz_dt}, we obtain
\[
\dfrac{d z_g}{d t} = - c_\nu\, z_0^{\nu + 1} \sum_{\ell = 0}^{3 g} \dfrac{a_{z,\ell}^{(1)}}{\left(\nu - (\nu-1) z_0\right)^{2g+\ell+1}}, \qquad a_{z,\ell}^{(1)} \in \mathbb{R},
\]
where each coefficients $a_{z,\ell}^{(1)}$ is a linear combinations of at most two of the $\left\{a_{z,k}^{(0)}\right\}$ coefficients. More generally, the $j$th derivative of $z_g$ with respect to $t$ may be written as \cite{bib:er}
\begin{equation}
    \label{eq:dz}
z_g^{(j)}(z_0) := (-1)^j\, \dfrac{d^j z_g}{d t^j} = c_\nu^j\, z_0^{j \nu + 1} \sum_{\ell = 0}^{3 g + j - 1} \dfrac{a_{z,\ell}^{(j)}}{\left(\nu - (\nu-1) z_0\right)^{2g+\ell+j}}, \qquad a_{z,\ell}^{(j)} \in \mathbb{R}.
\end{equation}
Similarly, one can write \cite{bib:er14}
\begin{equation}
    \label{eq:de}
e_g^{(j)}(z_0) := (-1)^j\, \dfrac{d^j e_g}{d t^j} = c_\nu^j\, z_0^{j \nu + 1} \sum_{\ell = 0}^{3 g + j - 4} \dfrac{b_{e,\ell}^{(j)}}{\left(\nu - (\nu-1) z_0\right)^{2g+\ell+j-1}}, \qquad j \ge 1,
\end{equation}
where the coefficients $b_{e,\ell}^{(j)}$ are scalars.
Both Equations \eqref{eq:dz} and \eqref{eq:de} provide expressions of the $j$th derivative of a rational function $G(z_0)$ in the form
\begin{equation}
    \label{eq:df}
G^{(j)}(z_0) := (-1)^j\, \dfrac{d^j G}{d t^j} =  c_\nu^j\, z_0^{j \nu + 1} \sum_{\ell = 0}^{\alpha +j} \dfrac{q_\ell^{(j)}}{\left(\nu - (\nu-1) z_0\right)^{\beta+\ell+j}}, \qquad j \ge j_G,
\end{equation}
where
\begin{align}
\label{eq:param}
\text{For } G &= z_g, & \alpha &= 3 g - 1, & \beta &= 2 g, & j_G &= 0. \nonumber\\ & \\
\text{For } G &= e_g, & \alpha &= 3 g - 4, & \beta &= 2 g - 1, & j_G &= 1. \nonumber
\end{align}

Our goal is to derive the vector recurrence relation satisfied by the coefficients $q_\ell^{(j)}$. Taking the derivative of \eqref{eq:df} with respect to $t$ and using \eqref{eq:dz_dt}, we have 
\begin{align*}
G^{(j+1)}(z_0)
&= - c_\nu^j \left((j \nu + 1) z_0^{j \nu} \sum_{\ell = 0}^{\alpha +j} \dfrac{q_\ell^{(j)}}{\left(\nu - (\nu-1) z_0\right)^{\beta+\ell+j}} + z_0^{j \nu + 1} \sum_{\ell = 0}^{\alpha +j} \dfrac{q_\ell^{(j)} (\nu-1)(\beta+\ell+j)}{\left(\nu - (\nu-1) z_0\right)^{\beta+\ell+j+1}}\right) \dfrac{d z_0}{d t}\\
&= c_\nu^{j+1} z_0^{(j+1) \nu+1} \left(\sum_{\ell = 0}^{\alpha +j} \dfrac{(j \nu + 1) q_\ell^{(j)}}{\left(\nu - (\nu-1) z_0\right)^{\beta+\ell+j+1}} + \sum_{\ell = 0}^{\alpha +j} \dfrac{q_\ell^{(j)} (\nu-1) (\beta+\ell+j) z_0}{\left(\nu - (\nu-1) z_0\right)^{\beta+\ell+j+2}}\right)
\end{align*}
Writing the second sum as
\begin{align*}
&\sum_{\ell = 0}^{\alpha +j} \dfrac{-q_\ell^{(j)} \big(\nu-(\nu-1) z_0-\nu\big)(\beta+\ell+j)}{\left(\nu - (\nu-1) z_0\right)^{\beta+\ell+j+2}}\\
= &\sum_{\ell = 0}^{\alpha +j} \dfrac{-q_\ell^{(j)} (\beta+\ell+j)}{\left(\nu - (\nu-1) z_0\right)^{\beta+\ell+j+1}} + \sum_{\ell = 0}^{\alpha +j} \dfrac{q_\ell^{(j)} \nu(\beta+\ell+j)}{\left(\nu - (\nu-1) z_0\right)^{\beta+\ell+j+2}}\\
= &\sum_{\ell = 0}^{\alpha +j} \dfrac{-q_\ell^{(j)} (\beta+\ell+j)}{\left(\nu - (\nu-1) z_0\right)^{\beta+\ell+j+1}} + \sum_{\ell = 1}^{\alpha +j+1} \dfrac{q_{\ell-1}^{(j)} \nu(\beta+\ell-1+j)}{\left(\nu - (\nu-1) z_0\right)^{\beta+\ell+j+1}},
\end{align*}
we see that the coefficients $q_\ell^{(j+1)}$ of $G^{(j+1)}(z_0)$ are related to those of $G^{(j)}(z_0)$ by
\[
q_\ell^{(j+1)}=(j \nu + 1 - \beta - \ell - j) q_\ell^{(j)} + q_{\ell-1}^{(j)} \nu(\beta+\ell-1+j), \quad 0 \le \ell \le \alpha + j + 1,
\]
with the convention that $q_{-1}^{(j)} = 0$. This may be rewritten as the following recurrence for $q_\ell^{(j)}$ in terms of the coefficients at order $j-1$
\begin{equation}
\label{eq:rec}
q_\ell^{(j)}= \nu \big(\beta+\ell+j-2 \big) q_{\ell-1}^{(j-1)} - \big(\beta + \ell - 1 -(\nu - 1)(j-1)\big) q_\ell^{(j-1)},\quad 0 \le \ell \le \alpha + j.
\end{equation}
With $\beta = 2 g$, the above relation appears in Lemma 3.2 of \cite{bib:er} for the derivatives of $z_g$, and for $\beta = 2 g - 1$, it corresponds to Lemma 5.2 of \cite{bib:er14} for the derivatives of $e_g$. We now define the semi-infinite column vector $V^{(j)}$ of coefficients of the partial fraction expansion of $G^{(j)}(z_0)/z_0^{2 j + 1}$ in powers of $\nu - (\nu-1) z_0$,
\begin{equation}
\label{eq:Vj}
V^{(j)}=c_\nu^j \left[q_0^{(j)}, q_1^{(j)}, \cdots, q_\ell^{(j)}, \cdots,  q_{\alpha+j}^{(j)}, 0, \cdots\right]^T, \quad j \ge j_G,
\end{equation}
where the first entry of $V^{(j)}$ corresponds to the lowest power of  $(\nu-(\nu-1) z_0)^{-1}$ in the partial fraction expansion of $\dfrac{G^{(j)}(z_0)}{z_0^{2 j + 1}}$, which by definition of $G^{(j)}$ is always $\beta + j$. The previous calculations are summarized in the following theorem.
\begin{theorem} \label{thm:M}
For $j \ge j_G$, let $V^{(j)}$ be the semi-infinite vector of coefficients of the partial fraction expansion of $R^{(j)}(z_0)= \dfrac{G^{(j)}(z_0)}{z_0^{2 j + 1}}$ in powers of $\nu - (\nu-1) z_0$, as defined in Equation \eqref{eq:Vj}. Then, there exists a family of semi-infinite sub-diagonal matrices $\Big\{M^{(n)}, n \ge 1 \Big\}$, such that
\[
V^{(j)} = c_\nu^{j-j_G} \left(\overleftarrow{\prod_{n=j_G+1}^j} M^{(n)}\right) \ V^{(j_G)},
\]
where
\[
\overleftarrow{\prod_{n=j_G+1}^j} M^{(n)} = M^{(j)} \, M^{(j-1)} \cdots M^{(j_G+1)}.
\]
In addition, the row $k$ ($k \ge 1$) and column $i$ ($i \ge 1$) entries
of $M^{(n)}$ are given by
\[
\left\{\begin{array}{l}
M^{(n)}[k,i] = 0 \qquad \text{for } i \ne k,\ k-1\\
M^{(n)}[k,k] = n (\nu - 1) - (\beta + k + \nu - 3)\\
M^{(n)}[k,k-1] = \nu(\beta + k + n - 3)
\end{array}.\right.
\]
\end{theorem}
\begin{rem}
We use square brackets to denote vector and matrix entries. Because indices on the entries of $M^{(n)}$ and $V^{(j)}$ start at $1$, the expressions in the above theorem are obtained by setting $\ell=k-1$ in Equation \eqref{eq:rec}.
\end{rem}

We note that there is a unique index $\delta^{(n)}=n (\nu - 1) - \beta + 2$ such that $M^{(n+1)}\Big[\delta^{(n)},\delta^{(n)}\Big] = 0$. This leads to a remarkable simplification in the 4-valent ($\nu = 2$) case, which we exploit in the next section. 

\section{The four-valent case}
\label{sec:4-valent}

As detailed in Appendix \ref{app:4-valent_dynamics}, a remarkable simplification occurs in the 4-valent case: the number of possibly non-zero entries of the vector $V^{(j)}$ is independent of $j$. This happens because when $\nu = 2$, the index $\delta^{(n)}$ introduced at the end of the previous section satisfies the property
\[
\delta^{(n+1)} = \delta^{(n)} + 1.
\] 
Consequently, the dynamics of $V^{(j)}$ may be described in terms of a product of operators defined on a vector space of fixed finite dimension. Appendix \ref{app:map_cts_dot_product} demonstrates that these operators may be simultaneously diagonalized, thereby leading to an explicit formulation for the non-recursive counts, as we will now summarize.

Recall from \eqref{eq:zg_parfrac} that the coefficients of the {\em full} Laurent polynomial of $z_g/z_0$ constitute a vector of length $5g-1$, \[
V_{z,2}^{(0)} = \left[0, \dots, 0,  a^{(0)}_{z,0}, \dots, a^{(0)}_{z,3g-1}\right]^T,
\]
in which the first $2g-1$ coefficients vanish. (By convention the constant term is not included here but it vanishes as well.) In \cite{bib:er} it is shown that $a^{(0)}_{z,3g-1} > 0$ for all $g > 0$. As just mentioned, in the 4-valent case the vector dynamics may be diagonalized, leading to a matrix operator at the $j^{th}$ stage whose row vector of column sums may be explicitly evaluated to be
\[
{\mathcal R}_z^{(j)} = \left[{\mathcal R}_z^{(j)}[1], \cdots, {\mathcal R}_z^{(j)}[5 g - 1]\right], \quad 
{\mathcal R}_z^{(j)}[n] = \frac{1}{2^{n-1}} \sum_{k=1}^{n} \left[\binom{n-1}{k-1} \prod_{\ell=0}^{j-1}
2 (2 \ell + k)\right].
\]
Then, the map counts are given by
\[
{\mathcal N}_{4,z}(g,j) = 12^j \ {\mathcal R}_z^{(j)} \cdot V_{z,2}^{(0)}.
\]
We note that ${\mathcal R}_z^{(j)}$
is completely independent of the genus.
A very similar result is obtained for
the counts in the regular 4-valent case, for $g > 1$. 
\[
{\mathcal N}_{4,e}(g,j) = 12^{j-1}\  
{\mathcal R}_e^{(j)} \cdot V_{e,2}^{(1)},
\]
where now
\[
V_{e,2}^{(1)} = \left[0, \dots, 0,  b^{(1)}_{e,0}, \dots, b^{(1)}_{e,3g-3}\right]^T,
\]
is the vector of coefficients in the Laurent polynomial for $- z_0^{-3} d e_g/ d z_0$  which has length $5g-5$, with the first $2g-3$ coefficients vanishing. 
${\mathcal R}_e^{(j)}$ differs from ${\mathcal R}_z^{(j)}$ only in that the product appearing in ${\mathcal R}_e^{(j)}[n]$ now runs from $1$ to $j-1$ rather than from $0$ to $j-1$, and that they are of different length. All of this is compactly summarized in the following theorem. 
\begin{theorem} \label{thm:gen_counts}
 Assume that $j_G \le j_0 = \beta - 1$, where $j_G$ and $\beta$ are defined in Equation \eqref{eq:param}. For $j \ge j_G$, let ${\mathcal N}_{4,z}(g,j)$ (resp.  ${\mathcal N}_{4,e}(g,j)$) be the number of 4-valent, 2-legged (resp. regular) maps with $j$ vertices that can be embedded in a surface of genus $g$. Each of these numbers is equal to $c_2^{j-j_G} = 12^{j-j_G}$ times the contraction of a row vector ${\mathcal R}^{(j)}$ with a column vector $X^{(j_G)}$ of initial coefficients, both of which have length $s$. The row vector ${\mathcal R}^{(j)}$ satisfies 
\[
{\mathcal R}^{(j)} = \left[{\mathcal R}^{(j)}[1], \cdots, {\mathcal R}^{(j)}[s]\right], \quad 
{\mathcal R}^{(j)}[n] = \frac{1}{2^{n-1}} \sum_{k=1}^{n} \left[\binom{n-1}{k-1} \prod_{\ell=j_G}^{j-1}
2 (2 \ell + k)\right]
\]
and $X^{(j_G)}$ is the vector of coefficients of the partial fraction decomposition of
\[
R^{(j_G)}(z_0) =  \dfrac{G^{(j_G)}(z_0)}{z_0^{2 j_G + 1}}
\]
in powers of $(z_0 - 2)$.
The numbers ${\mathcal N}_{4,z}(g,j)$ and ${\mathcal N}_{4,e}(g,j)$ only differ through the value of $s$, with
\[
s_{z_g} = 5 g - 1 \quad \text{and} \quad s_{e_g} = 5 g - 5,
\]
the parameters $j_G$ and $\beta$ introduced in Equation \eqref{eq:param}, and the initial vector of coefficients $X_g^{(j_G)}$.
\end{theorem}

Note that the vector ${\mathcal R}^{(j)}$ may be viewed as the truncation to a vector of length $s$ of a universal, semi-infinite vector 
\[
\left[{\mathcal R}^{(j)}[1], \cdots, {\mathcal R}^{(j)}[n], \cdots \right], \quad 
{\mathcal R}^{(j)}[n] = \frac{1}{2^{n-1}} \sum_{k=1}^{n} \left[\binom{n-1}{k-1} \prod_{\ell=j_G}^{j-1}
2 (2 \ell + k)\right]
\]
whose entries are {\em independent of the genus $g$}. This description provides interesting insight about the structure of the general formulas of Theorem \ref{th:counts_hypergeom} in the case when $\nu = 2$.

We may now apply Theorem \ref{thm:gen_counts} to obtain closed-form expressions for the number of 4-valent maps for given values of $g$. These require knowledge of $z_g(z_0)$ and $e_g(z_0)$, in order to find the initial vector of coefficients $X_g^{(j_G)}$, as well as the value of the initial index $j_G$. Given the nature of the entries of the row vector ${\mathcal R}^{(j)}$, we will make use of the following formulas, which can be established by direct calculation for $p \ge 0$.
\begin{equation} \label{eq:prods}
d_{2 p,0}^{(j)} = \prod_{\ell=0}^{j-1} (4 \ell + 4 p) = 4^j \dfrac{(j+p-1)!}{(p-1)!}, \qquad 
d_{2 p + 1,0}^{(j)} = \prod_{\ell=0}^{j-1} 2 (2 \ell + 1 + 2 p) = \dfrac{(2 j + 2 p)!\, p!}{(2 p)!\, (j+p)!}.
\end{equation}
Since Theorem \ref{thm:gen_counts} requires that $j_G \le j_0 = \beta - 1$, we obtain 
\[
j_G \le \beta - 1 \Longleftrightarrow \left\{\begin{matrix} 0 \le 2 g - 1 & \text{for } z_g \\ 1 \le 2 g - 2 & \text{for } e_g \end{matrix} \right. \Longleftrightarrow \left\{\begin{matrix} g \ge 1/2 & \text{for } z_g \\ g \ge 3/2 & \text{for } e_g \end{matrix} \right. .
\]
Consequently, the method described in this section provides closed form expressions for ${\mathcal N}_{4,z}(g,j)$ when $g \ge 1$ (see Table \ref{tab:counts_z}), and for ${\mathcal N}_{4,e}(g,j)$ when $g \ge 2$ (Table \ref{tab:counts_e}). 

\subsection{Map count for $g=1$}
For $g=1$, the method only applies to ${\mathcal N}_{4,z}(1,j)$ (see Appendix \ref{app:e1} for a calculation of ${\mathcal N}_{4,e}(1,j)$ in terms of hypergeometric functions). Then $\alpha = 3 g - 1 = 2$, $\beta = 2 g = 2$, and $s = \alpha + \beta = 4$. The row vector ${\mathcal R}^{(j)}$, which we denote by ${\mathcal R}_0^{(j)}$ to emphasize that the products in $d_m^{(j)}$ start at $j_G = 0$, may be truncated to its first $s=4$ entries, leading to
\[
{\mathcal R}_0^{(j)}[1,2,3,4] = \left[d_1^{(j)},  \dfrac{d_1^{(j)}+d_2^{(j)}}{2},  \dfrac{d_1^{(j)} + 2 d_2^{(j)} + d_3^{(j)}}{4},  \dfrac{d_1^{(j)} + 3 d_2^{(j)} + 3 d_3^{(j)} + d_4^{(j)}}{8} \right].
\]
The rational function $z_1$ is given by \cite{bib:emp08}
\[
z_{1}(z_0) = 
\frac{z_{0} \left(z_{0}-1\right) \left(\displaystyle -\frac{2}{3}+\frac{2 z_{0}}{3}\right)}{\left(2-z_{0}\right)^{4}} = z_0 \left(\dfrac{2/3}{(2-z_0)^2}+ \dfrac{-4/3}{(2-z_0)^3} + \dfrac{2/3}{(2-z_0)^4}\right).
\]
Since $j_0 = \beta - 1 = 1$, we have $j_0 > j_G = 0$ and, as explained in Remark \ref{rem:jG_less_j0},  we need to augment the initial set of 3 coefficients in the partial fraction expansion of $z_1(z_0)/z_0$ with $j_0 - j_G = 1$ zero, on the left. The number ${\mathcal N}_{4,z}(1,j)$ is obtained by contracting ${\mathcal R}_0^{(j)}[1,2,3,4]$ with
\[
V_{z,1}^{(0)} = \displaystyle \frac{1}{3} [0, 2, -4, 2]^T.
\]
Specifically,
\begin{align*}
{\mathcal N}_{4,z}(1,j) &= c_2^{j-j_G} {\mathcal R}_0^{(j)}[1,2,3,4] \cdot V_{z,1}^{(0)} \\ 
& = c_2^j \left(\dfrac{2}{3} \dfrac{d_1^{(j)}+d_2^{(j)}}{2} - \dfrac{4}{3} \dfrac{d_1^{(j)} + 2 d_2^{(j)} + d_3^{(j)}}{4} + \dfrac{2}{3} \dfrac{d_1^{(j)} + 3 d_2^{(j)} + 3 d_3^{(j)} + d_4^{(j)}}{8}\right)\\
&= \dfrac{12^{j}}{12} \left(d_1^{(j)} - d_2^{(j)} - d_3^{(j)} + d_4^{(j)} \right).
\end{align*}
Since
\[
d_1^{(j)} = \dfrac{(2j)!}{j !},
\quad d_2^{(j)} = 4^j\, j!,
\quad d_3^{(j)} = \dfrac{(2j+2)!}{2 (j+1)!},
\quad d_4^{(j)} = 4^j (j+1)!
\]
we obtain 
\[
d_1^{(j)} - d_3^{(j)} = - 2 j \dfrac{(2j)!}{j !}, \qquad d_4^{(j)} - d_2^{(j)} = 4^j\, j\, j!
\]
and thus
\[
{\mathcal N}_{4,z}(1,j)= 12^j\, j \left(\dfrac{4^j j!}{12} -  \dfrac{(2j)!}{6 j!}\right)
\]
for $j \ge 1$, as indicated in Table \ref{tab:counts_z}. 

\subsection{Map counts for $2 \le g \le 7$}\label{sec:mc2-7}
For $g \ge 2$, the method applies to both ${\mathcal N}_{4,z}(g,j)$ and ${\mathcal N}_{4,e}(g,j)$ for all values of $j \ge 1$. As previously indicated, an explicit formulation requires knowledge of $z_g(z_0)$ and $e_g(z_0)$. In \cite{bib:elt22b}, we derived these functions for genera up to $g = 7$. In all cases, $j_0 > j_G$ and the initial vector $V^{(j_G)}$, of size $s = \alpha + \beta$, has its first $j_0 - j_G$ entries equal to zero. We list these vectors in Appendices \ref{app:zg} (for ${\mathcal N}_{4,z}(g,j)$) and \ref{app:eg} (for ${\mathcal N}_{4,e}(g,j)$). The  result of contracting each of these vectors with either ${\mathcal R}_0^{(j)}$ (when $j_G = 0$) or ${\mathcal R}_1^{(j)}$ (when $j_G = 1$), truncated to length $s$, was subsequently simplified, with the help of Maple \cite{bib:maple}, into the expressions provided in Tables \ref{tab:counts_z} and \ref{tab:counts_e}.

\begin{figure}
\centering
\includegraphics[width= .95\linewidth]{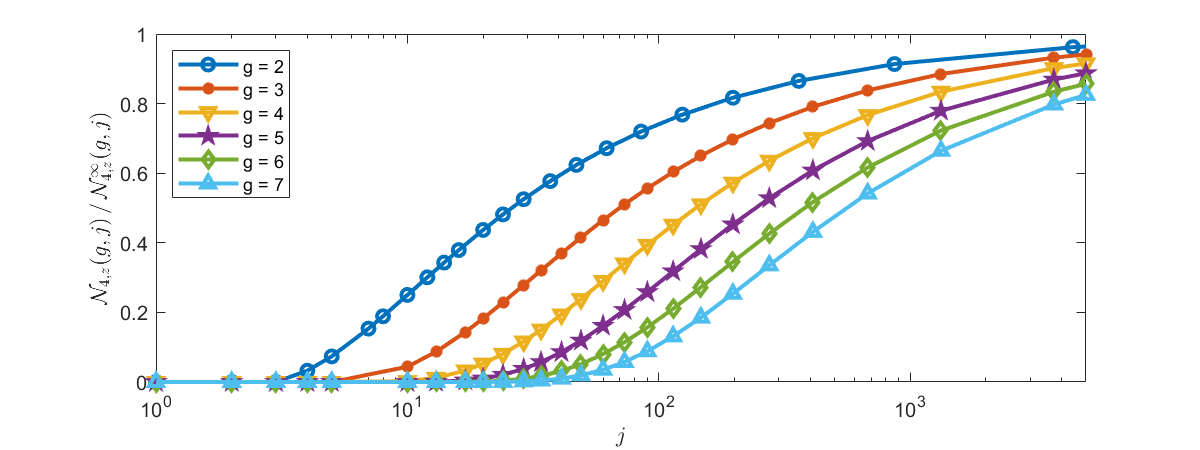}\\
\includegraphics[width= .95\linewidth]{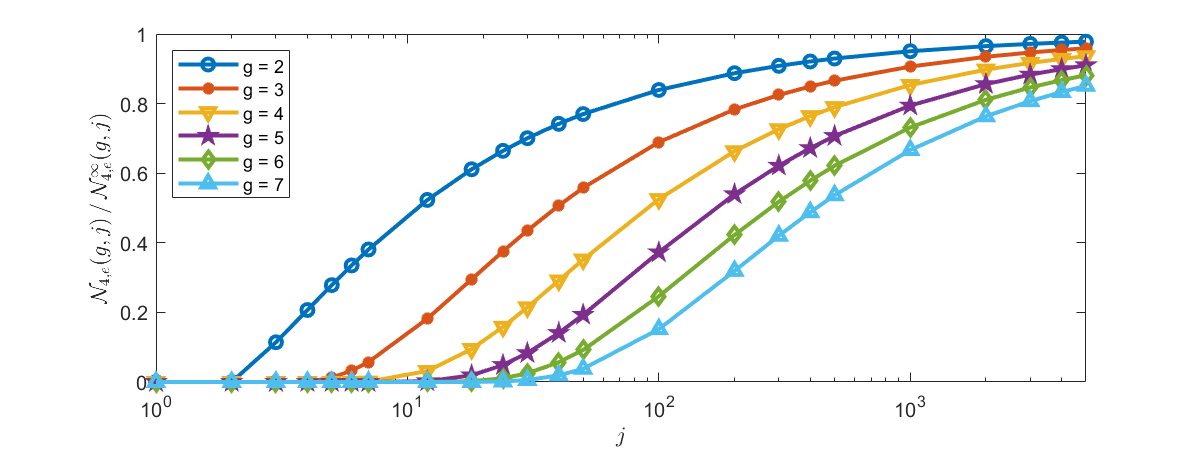}
\caption{Top: Convergence of the ratios ${\mathcal N}_{4,z}(g,j)\, / \,{\mathcal N}_{4,z}^\infty(g,j)$ towards $1$ as $j \to \infty$. Bottom: Similar plot for ${\mathcal N}_{4,e}(g,j)\, / \,{\mathcal N}_{4,e}^\infty(g,j)$. The asymptotic term ${\mathcal N}_{4,z}^\infty(g,j)$ is defined in Equation \eqref{zg-asymp} and ${\mathcal N}_{4,e}^\infty(g,j)$ is given in Equation \eqref{eg-asymp}.}
\label{fig:convergence}
\end{figure}
The top panel of Figure \ref{fig:convergence} illustrates the convergence of ${\mathcal N}_{4,z}(g,j)$, as given in Table \ref{tab:counts_z}, towards its dominant term ${\mathcal N}_{4,z}^\infty(g,j)$ as $j \to \infty$, for different values of $g \ge 2$. Here, ${\mathcal N}_{4,z}^\infty(g,j)$ is defined by
\[
{\mathcal N}_{2\nu,z}(g,j)= {\mathcal N}_{2\nu,z}^\infty(g,j) \ (1 + o(1)) \quad \text{as} \quad j \to \infty
\]
where
\begin{equation}
\label{zg-asymp}
{\mathcal N}_{2\nu,z}^\infty(g,j) = j!\, \frac{\nu}{\nu-1} \frac{a_{z,3g-1}^{(0)}(g,\nu) }{ \left(\sqrt{2\nu / (\nu - 1)}\right)^{5g-1}} \dfrac{j^{\frac{5g-3}{2}}}{\Gamma\left(\frac{5g-1}{2} \right)}
s_c^{-j}, \quad 
s_c = \frac{(\nu - 1)^{\nu -1}}{c_\nu\, \nu^\nu},
\end{equation}
and $c_\nu$ is given in Equation \eqref{eq:dz_dt}. The above asymptotic formula was established in \cite{bib:er}\footnote{Unfortunately, the term $\sqrt{2\nu/(\nu-1)}$ in the denominator of ${\mathcal N}_{2\nu,z}^\infty(g,j)$ was incorrectly written as $ \sqrt{2\nu}\, (\nu - 1)$, both in \cite{bib:er} and \cite{bib:ew}. A similar typo appears in the expression for ${\mathcal N}_{2\nu,e}^\infty(g,j)$ given in \cite{bib:ew}.}. A similar result for ${\mathcal N}_{2\nu,e}(g,j)$ appeared in \cite{bib:ew}:
\[
{\mathcal N}_{2\nu,e}(g,j) = {\mathcal N}_{2\nu,e}^\infty(g,j)\ (1+ o(1)) \quad \text{as} \quad j \to \infty,
\]
where
\begin{equation}
\label{eg-asymp}
{\mathcal N}_{2\nu,e}^\infty(g,j) = j!\, \frac{b^{(0)}_{e,3g-3}(g,\nu)}{\left(\sqrt{2 \nu /(\nu-1)}\right)^{5g-5}} \frac{j^{\frac{5g-7}{2}}}{\Gamma\left( \frac{5g-5}{2}\right)} s_c^{-j}, \qquad
b_{e,3g-3}^{(0)}(g,\nu) = \frac{a_{z,3g-1}^{(0)}(g,\nu)}{(5g-5)(5g-3) \nu^2}  > 0.
\end{equation}
The behavior of the ratio ${\mathcal N}_{2\nu,e}(g,j)\, / \, {\mathcal N}_{2\nu,e}^\infty(g,j)$ is displayed in the bottom panel of Figure \ref{fig:convergence}, showing convergence towards $1$ for large values of $j$.

\subsection{Reconciling the two formulations}
We can compare the general formulation of ${\mathcal N}_{2\nu,z}(g,j)$ given in Section \ref{sec:2} with the result of Theorem \ref{thm:gen_counts}, specific to the $\nu = 2$ case, since both involve the coefficients $a_{z,\ell}^{(0)}$.
\begin{align*}
{\mathcal N}_{4,z}(g,j) &= j!\, c_\nu^j \sum_{\ell = 0}^{3g-1} a_{z,\ell}^{(0)} {(2g-2) + (\ell+j)  \choose j} \ {_2}F_1 \left( \genfrac{}{}{0pt}{}
{-j,\ -2 j}{2-2g-(\ell+j)}; -1 \right)\\
&= c_\nu^j \sum_{\ell=1}^{5 g - 1} {\mathcal R}_0^{(j)}[\ell] \ X^{(0)}[\ell] = c_\nu^j \sum_{\ell = 2 g}^{5 g - 1} a_{z,\ell-2g}^{(0)} \ {\mathcal R}_0^{(j)}[\ell] = c_\nu^j \sum_{\ell = 0}^{3 g - 1} a_{z,\ell}^{(0)} \ {\mathcal R}_0^{(j)}[\ell+2 g].\\
\Longrightarrow 0 & = 
\sum_{\ell=0}^{3 g - 1} a_{z,\ell}^{(0)} \left(j!\, {(2g-2) + (\ell+j)  \choose j} \ {_2}F_1 \left( \genfrac{}{}{0pt}{}
{-j,\ -2 j}{2-2g-(\ell+j)}; -1 \right) \right. \\*
& \left. \qquad \qquad \qquad - \dfrac{1}{2^{\ell+2g-1}} \sum_{k=1}^{\ell+2 g} {\ell + 2 g - 1 \choose k-1} \prod_{m=0}^{j-1} 2 (2 m + k)
\right).
\end{align*}
Although equality of the two linear combinations here does not require equality of their coefficients, we checked that the following identity was true for a range of integer values of $\ell$, $g$, and $j$, specifically $0 \le \ell \le 20$ and $1 \le g,\, j \le 20$. We leave this statement as a conjecture for future analysis.
\begin{conjecture}
For integer values of $g \ge 1$, $\ell \ge 0$, and $j \ge 1$, the following identity is true
\[
j!\, 2^{\ell + 2 g - 1} {(2g-2) + (\ell+j)  \choose j} \ {_2}F_1 \left( \genfrac{}{}{0pt}{}
{-j,\ -2 j}{2-2g-(\ell+j)}; -1 \right)
= \sum_{k=1}^{\ell+2 g} {\ell + 2 g - 1 \choose k-1} \prod_{m=0}^{j-1} 2 (2 m + k).
\]
\end{conjecture}

\section{Conclusions} \label{sec:7}

In this paper we have considered the problem of finding explicit closed-form expressions of map counts on a surface of genus $g$. Such expressions enable one to immediately read off the desired counts for maps with $j$ vertices ($j$ arbitrary) without any iterative calculations required. This is what we have referred to as non-recursive
counts. We have explicitly solved this problem in the case of 4-valent maps, for genera $g \le 7$, as demonstrated in Tables \ref{tab:counts_z} and \ref{tab:counts_e}. As mentioned in the introduction, examples of such non-recursive counts did already exist in the literature, but for rather low values of the genus ($g \leq 3$). 

We have presented two different approaches to formulating the map counts. The first, which applies Cauchy's integration formula to the relevant generating function, led to an explicit expression for ${\mathcal N}_{2\nu,e}(1,j)$, for arbitrary even valence. This approach also provided general expressions (regardless of the value of $g$) for counts of $2\nu$-valent maps, in terms of linear combinations of the coefficients of the partial fraction expansions of $z_g/z_0$ and $e_g$. The second approach looked at the dynamics of the coefficients of the partial fraction expansion associated with successive derivatives of the generating functions, with the order of the derivatives playing the role of time. In the case when $\nu = 2$, we were able to solve the corresponding vector difference equations in terms of a universal, genus-independent row vector: truncating this vector to a row of finite, genus-dependent length, and contracting the result with a column vector of initial partial fraction coefficients, led to the desired counts. The advantage of the second approach is that it produces expressions for the counts in terms of elementary combinatorial functions of the number of vertices $j$. The connection between these two seemingly different methods led us to conjecture an identity involving hypergeometric functions.

Both approaches are able to generate closed expressions for non-recursive counts effectively and rather quickly for comparatively high values of the genus, once the initial coefficients are known. A general framework to derive these coefficients is described in a separate manuscript \cite{bib:elt22b}, which itself builds on \cite{bib:elt22}. Consequently, the results of the present work, together with the techniques developed in \cite{bib:elt22} and \cite{bib:elt22b}, completely address the problem of establishing closed-form non-recursive formulas for the number of even-valence maps with an arbitrary number of vertices, embedded in surfaces of arbitrary genus. The remainder of this section is a brief discussion of some questions that now appear open for fruitful investigation thanks to the results of the present manuscript.

First, we note that the method for evaluating contour integrals developed here can easily be generalized to other rational functions of $z_0$. This could, for instance, be used to calculate more general correlation functions arising in random matrix theory. In addition, there are by now a number of instances of the application of perturbative methods of statistical physics to the analysis of geometric partition functions. These lead to asymptotic expansions whose coefficients are generating functions for various topological or geometric invariants. They include applications in knot theory \cite{bib:w88}, hyperbolic surface theory \cite{bib:m08}, and symplectic manifold invariants \cite{bib:k96}. It will be interesting to see if our methods may be applied, to some extent, to these other contexts. For instance, we have found strong numerical evidence that
\[
\chi(\Gamma_g) =  \sum_{\ell=0}^{3g-3} b_{e,\ell}^{(0)}
\] 
where $b_{e,\ell}^{(0)}$ are the Laurent coefficients for $e_g(z_0)$
and $\chi(\Gamma_g)$ is the orbifold Euler characteristic for the mapping class group $\Gamma_g$ (equivalently, the
moduli space of genus $g$ Riemann surfaces) \cite{bib:hz}.
In a similar vein, non-recursive counts of maps with boundaries appear in a number of places in the literature going back to Tutte \cite{bib:tu} in the case of $g=0$, with more recent examples for $g \leq 2$ being provided in \cite{bib:ebook}. Such enumerations play a role in recent work on Jackiw-Teitelboim quantum gravity \cite{bib:sw} and our methods may be relevant to these questions as well.

Second, it is natural to ask how much of the analysis of the dynamics of the vector $V^{(j)}$ carried out in Sections \ref{sec:rec_relations} and \ref{sec:4-valent} can be continued to the case when $\nu \ne 2$. Certainly the ordered product representation, displayed in Theorem \ref{thm:M}, of the solution to the vector difference equation is quite general. The obstacle arises in the absence of an evident reduction of the dynamics to a vector of fixed length when $\nu \ne 2$. Furthermore, and this is partially a consequence of the prior obstacle, one does not in general have a simultaneous diagonalization of the operators in the ordered product representation. However, our investigations, which are beyond the scope of the present manuscript, show that one can diagonalize each factor of that product. The result is a product in which the respective diagonalizations are interlaced by a common operator that turns out to be an explicit Toeplitz operator.

In a slightly different direction, one can consider the case of odd valence \cite{bib:bd13, bib:ep}. Structures analogous to what we have studied for even valence indeed exist for triangulations (dual to trivalent maps). In particular, rational expressions for $z_g$ and $e_g$ as functions of $z_0$ are available. In \cite{bib:elt22b} these structures are used to calculate map counts. As mentioned in the introduction, non-recursive map counts in the 3-valent case are known for $g=0$ and $g=1$ \cite{bib:bd13}. It is a natural goal to extend this to higher genus and we intend to consider these questions in future work. 

\appendix

\section{Non-recursive counts of $e_1$} \label{app:e1}
In this appendix, we provide a direct evaluation of ${\mathcal N}_{2\nu,e}(1,j)$ starting from the contour integral representation given in \cite{bib:er},
\begin{equation} \label{e1_count}
   {\mathcal N}_{2\nu,e}(1,j) =
   \frac{j!\, c_\nu^j}{2\pi i}
   \oint \frac{e_1(\eta)}{\eta^{j+1}} d\eta,
\end{equation}
where the integral is taken on a small loop surrounding the origin, $e_1(\eta) = -\dfrac{1}{12} \log\left(\nu - (\nu-1) z_0\right)$ and the relation between $\eta$ and $z_0$ is given by the string equation
\begin{equation*}
    1 = z_0 - \eta z_0^\nu. 
\end{equation*}
\begin{theorem} The number of $2 \nu$-regular maps with $j \ge 1$ vertices that can be embedded in a surface of genus 1 is given by
\begin{eqnarray*}
{\mathcal N}_{2\nu,e}(1,j) &=&
\dfrac{j!\, c_\nu^j}{12}
\left( (\nu - 1) {\nu j-1 \choose j-1}\, {_3}F_2 \left( \genfrac{}{}{0pt}{}
{1,\ 1,\ 1-j}{2,(\nu-1)j+1}; 1-\nu \right) \right. \nonumber \\
&& \qquad \quad \left. - (\nu - 1)^2 {\nu j-1 \choose j-2}\ {_3}F_2
\left( \genfrac{}{}{0pt}{}
{1,\ 1,\ 2-j}{2,(\nu-1)j+2}; 1-\nu \right)\right),
\end{eqnarray*}
where ${_3}F_2$ is the generalized hypergeometric function \cite{bib:nist}.
\end{theorem}
\begin{proof}
Using the string equation, the contour integral \eqref{e1_count} is rewritten as
\begin{eqnarray*}
&&  -\dfrac{j!\, c_\nu^j}{12} \frac{1}{2\pi i}
   \oint \frac{\log(\nu - (\nu - 1) z) (\nu - (\nu - 1) z)}{(z-1)^{j+1}} z^{\nu j - 1} dz\\
   &=& -\dfrac{j!\, c_\nu^j}{12} \frac{1}{2\pi i}
   \oint \frac{\log(1- (\nu - 1) u) (1 - (\nu - 1) u)}{u^{j+1}} (1+u)^{\nu j - 1} du \\
   &=& -\dfrac{j!\, c_\nu^j}{12} \sum_{m=0}^{\nu j -1} {\nu j - 1 \choose m} \frac{1}{2\pi i}
   \oint \log(1- (\nu - 1) u) (1 - (\nu - 1) u) u^{m - j -1} du \\
   &=& -\dfrac{j!\, c_\nu^j}{12} \sum_{m=0}^{\nu j -1} {\nu j - 1 \choose m} \frac{1}{2\pi i}
   \oint \log(1- w) (1 - w) \frac{w^{m - j -1}}{(\nu-1)^{m-j}} dw \\
   &=& \dfrac{j!\, c_\nu^j}{12} \sum_{m=0}^{\nu j -1} \frac1{(\nu -1)^{m-j}}{\nu j - 1 \choose m} \frac{1}{2\pi i}
   \oint (1 - w) \sum_{k=1}^\infty  \frac{w^k}{k} 
   w^{m - j -1} dw \\
   &=& \dfrac{j!\, c_\nu^j}{12} \sum_{m=0}^{\nu j -1}\sum_{k=1}^\infty \frac1{k (\nu -1)^{m-j}}{\nu j - 1 \choose m} \frac{1}{2\pi i}
   \oint   (1 - w)
   w^{k+ m - j -1} dw \\
   &=& \dfrac{j!\, c_\nu^j}{12} \sum_{m=0}^{\nu j -1}\sum_{k=1}^\infty \frac1{k (\nu -1)^{m-j}}{\nu j - 1 \choose m}
   \left( \delta_{k, j-m}\Big|_{m < j} - \delta_{k, j-m-1}\Big|_{m < j-1} \right)\\
   &=& \dfrac{j!\, c_\nu^j}{12} \left( \sum_{m=0}^{ j -1}
   \frac1{(\nu -1)^{m-j}}
   \frac{1}{j-m} {\nu j - 1 \choose m} - \sum_{m=0}^{j-2}
   \frac1{(\nu -1)^{m-j}} \frac{1}{j-m-1} 
   {\nu j - 1 \choose m} \right)\\
   &=& \dfrac{j!\, c_\nu^j}{12} \left( \sum_{m=0}^{ j -1}
   \frac1{(\nu -1)^{-m-1} }
   \frac{1}{m+1} {\nu j - 1 \choose j-m-1} - \sum_{m=0}^{j-2}
   \frac1{(\nu -1)^{-m-2} } \frac{1}{m+1} 
   {\nu j - 1 \choose j-m-2} \right)
\end{eqnarray*}
In the second line we have made the substitution $u = z - 1$. In the third line we apply the binomial expansion since $j \ge 1$ and in the fifth line the logarithm is expanded in its Taylor series since it is analytic in a small neighborhood of $w = 0$ containing the contour. In the last line, the summation is reversed. Rewriting
\begin{align*}
{\nu j - 1 \choose j-m-1} & = \dfrac{(\nu j - 1)!}{(j-1)!\, (\nu j-j)!} \dfrac{(j-1)^{\underline{m}}}{(\nu j - j + 1)^{\overline{m}}} = {\nu j - 1 \choose j-1} \dfrac{(j-1)^{\underline{m}}}{((\nu-1) j + 1)^{\overline{m}}}\\
{\nu j - 1 \choose j-m-2} & = \dfrac{(\nu j - 1)!}{(j-2)!\, (\nu j-j+1)!} \dfrac{(j-2)^{\underline{m}}}{(\nu j - j + 2)^{\overline{m}}} = {\nu j - 1 \choose j-2} \dfrac{(j-2)^{\underline{m}}}{((\nu -1) j + 2)^{\overline{m}}},
\end{align*}
where $x^{\overline{m}} = x (x+1) \cdots (x+m-1)$ is the rising factorial and $x^{\underline{m}} = x (x-1) \cdots (x-m+1)$ is the falling factorial of $x$, we have
\begin{align*}
{\mathcal N}_{2\nu,e}(1,j) = &\dfrac{j!\, c_\nu^j}{12} \left( \sum_{m=0}^{ j -1}
   \frac1{(\nu -1)^{-m-1} }
   \frac{1}{m+1} {\nu j - 1 \choose j-1} \dfrac{(j-1)^{\underline{m}}}{((\nu-1) j + 1)^{\overline{m}}} \right. \\
& \left. \qquad \qquad - \sum_{m=0}^{j-2}
   \frac1{(\nu -1)^{-m-2} } \frac{1}{m+1} 
   {\nu j - 1 \choose j-2} \dfrac{(j-2)^{\underline{m}}}{((\nu-1) j + 2)^{\overline{m}}} \right)\\
= & \dfrac{j!\, c_\nu^j}{12} \left( {\nu j - 1 \choose j-1} \sum_{m=0}^{ j -1}(\nu -1)^{m+1} \frac{m!}{(m+1)!}  \dfrac{(j-1)^{\underline{m}}}{((\nu-1) j + 1)^{\overline{m}}} \right. \\
& \left. \qquad \qquad - {\nu j - 1 \choose j-2} \sum_{m=0}^{j-2} (\nu -1)^{m+2} \frac{m!}{(m+1)!} \dfrac{(j-2)^{\underline{m}}}{((\nu-1) j + 2)^{\overline{m}}}  \right)\\
= & \dfrac{j!\, c_\nu^j}{12} \left( {\nu j - 1 \choose j-1} \sum_{m=0}^{\infty} \frac{(1)^{\overline{m}}\, (1)^{\overline{m}}}{(2)^{\overline{m}}}  \dfrac{(1-j)^{\overline{m}}}{((\nu-1) j + 1)^{\overline{m}}} \dfrac{-(1-\nu)^{m+1}}{m!}\right. \\
& \left. \qquad \qquad - {\nu j - 1 \choose j-2} \sum_{m=0}^{\infty} \frac{(1)^{\overline{m}}\,(1)^{\overline{m}}}{(2)^{\overline{m}}\,} \dfrac{(2-j)^{\overline{m}}}{((\nu-1) j + 2)^{\overline{m}}} \dfrac{(1-\nu)^{m+2}}{m!} \right)\\
= &
\dfrac{j!\, c_\nu^j}{12}
\left( (\nu - 1) {\nu j-1 \choose j-1}\, {_3}F_2 \left( \genfrac{}{}{0pt}{}
{1,\ 1,\ 1-j}{2,(\nu-1)j+1}; 1-\nu \right) \right. \nonumber \\
& \qquad \quad \left. - (\nu - 1)^2 {\nu j-1 \choose j-2}\ {_3}F_2
\left( \genfrac{}{}{0pt}{}
{1,\ 1,\ 2-j}{2,(\nu-1)j+2}; 1-\nu \right)\right),
\end{align*}
where ${_3}F_2$ is the generalized hypergeometric function. In the penultimate line of the above calculation, it is understood that the ratios in the infinite sums are set to zero when their numerator vanishes through the rising factorial $(1-j)^{\overline{m}}$ or $(2-j)^{\overline{m}}$. Maple \cite{bib:maple} uses the same convention when evaluating these functions.
\end{proof}

\section{Contour Integral Representations in terms of Hypergeometric Functions} \label{app:D}
In this appendix, we use partial fraction expansions of the rational forms of $z_g$ and $e_g$ as functions of $z_0$ to express ${\mathcal N}_{2\nu,z}(g,j)$ and ${\mathcal N}_{2\nu,e}(g,j)$ (for $g > 1$) in terms of hypergeometric functions. We first consider the case of $z_g$ and make use of Equation \eqref{eq:zg_parfrac}, written here as
\[
z_g (z_0) = z_0 \sum_{\ell = 0}^{3 g-1} \dfrac{a_\ell^{(0)}(g,\nu)}{\left(\nu - (\nu-1) z_0\right)^{2g+\ell}}, \qquad a_\ell^{(0)}(g,\nu) \in \mathbb{R}.
\]
where we have emphasized the $g$- and $\nu$-dependence of the $a_\ell^{(0)}(g,\nu)$ coefficients. For clarity, this dependence is dropped in the calculations below.
\begin{eqnarray*}
{\mathcal N}_{2\nu,z}(g,j) &=& 
\dfrac{j!\, c_\nu^j}{2\pi i}\oint \dfrac{z_g(\eta)}{\eta^{j+1}} d\eta\\
&=& \dfrac{j!\, c_\nu^j}{2\pi i} \oint \frac{(\nu - (\nu-1) z) }{(z-1)^{j+1}} z^{\nu j} 
\frac{z_g}{z} dz\\
&=& j!\, c_\nu^j \sum_{\ell = 0}^{3g-1} a_\ell^{(0)}\frac1{2\pi i} \oint \frac{(\nu - (\nu-1) z) }{(z-1)^{j+1} (\nu - (\nu-1) z)^{2g+\ell}} z^{\nu j} dz\\
&=& j!\, c_\nu^j \sum_{\ell = 0}^{3g-1} a_\ell^{(0)}\frac1{2\pi i} \oint \frac{z^{\nu j} }{(z-1)^{j+1} (\nu - (\nu-1) z)^{2g+\ell-1}}  dz\\
&=& j!\, c_\nu^j \sum_{\ell = 0}^{3g-1} a_\ell^{(0)}\frac1{2\pi i} \oint \frac{(1+u)^{\nu j} }{u^{j+1} (1 - (\nu-1) u)^{2g+\ell-1}}  du\\
&=& j!\, c_\nu^j \sum_{\ell = 0}^{3g-1} a_\ell^{(0)}\frac1{2\pi i} \oint \frac{\sum_{n=0}^\infty\left({2g-1+\ell \choose n}\right) ((\nu-1)u)^n \sum_{m=0}^{\nu j} { \nu j \choose m} u^m}{u^{j+1} }  du\\
&=& j!\, c_\nu^j \sum_{\ell = 0}^{3g-1} a_\ell^{(0)}\sum_{n=0}^\infty (\nu-1)^n \sum_{m=0}^{\nu j}  \left({2g-1+\ell \choose n}\right)   { \nu j \choose m} \frac1{2\pi i} \oint u^{m +n - j -1}  du\\
&=& j!\, c_\nu^j \sum_{\ell = 0}^{3g-1} a_\ell^{(0)}\sum_{n=0}^\infty (\nu-1)^n \sum_{m=0}^{\nu j}  \left({2g-1+\ell\choose n}\right)   { \nu j \choose m} \delta_{n+m, j}\Big|_{m,n \geq 0}\\
&=& j!\, c_\nu^j \sum_{\ell = 0}^{3g-1} a_\ell^{(0)}  \sum_{m=0}^{ j} (\nu-1)^{j-m} \left({2g-1+\ell \choose j-m}\right) { \nu j \choose m} \\
&=& j!\, c_\nu^j \sum_{\ell = 0}^{3g-1} a_\ell^{(0)} \sum_{m=0}^{ j} (\nu-1)^{j-m} {(2g-2) + (\ell+j) -m \choose j-m}   { \nu j \choose m}\\
&=& j!\, c_\nu^j \sum_{\ell = 0}^{3g-1} a_\ell^{(0)} {(2g-2) + (\ell+j)  \choose j}  \sum_{m=0}^{ j} \frac{(\nu-1)^{j-m}  { j \choose m}}{{(2g-2) + (\ell+j) \choose m}}  { \nu j \choose m}  \\
&=& j!\, c_\nu^j (\nu-1)^{j} \sum_{\ell = 0}^{3g-1} a_\ell^{(0)} {(2g-2) + (\ell+j)  \choose j}  \sum_{m=0}^{ j} \frac{ (\nu j)^{\underline{m}}  { j \choose m}}{{(2g-2) + (\ell+j) \choose m}}  \frac{(\nu-1)^{-m}}{m!}
\\
&=& j!\, c_\nu^j (\nu-1)^{j} \sum_{\ell = 0}^{3g-1} a_\ell^{(0)} {(2g-2) + (\ell+j)  \choose j}  \sum_{m=0}^{ j} \frac{ (\nu j)^{\underline{m}}  ( j)^{\underline{m}}}{((2g-2) + (\ell+j) )^{\underline{m}}}  \frac{(\nu-1)^{-m}}{m!}\\
&=& j!\, c_\nu^j (\nu-1)^{j} \sum_{\ell = 0}^{3g-1} a_\ell^{(0)} {(2g-2) + (\ell+j)  \choose j}  \sum_{m=0}^{ \infty} \frac{ (-\nu j)^{\overline{m}}  ( -j)^{\overline{m}}}{((2-2g) -(\ell+j) )^{\overline{m}}}  \frac{(1-\nu)^{-m}}{m!}\\
&=& j!\, c_\nu^j (\nu-1)^{j} \sum_{\ell = 0}^{3g-1} a_\ell^{(0)} {(2g-2) + (\ell+j)  \choose j} 
\ {_2}F_1
\left( \genfrac{}{}{0pt}{}
{-j,\ -\nu j}{2-2g-(\ell+j)}; \frac{1}{1-\nu} \right)
\end{eqnarray*}
where ${_2}F_1$ denotes the Gauss Hypergeometric function \cite{bib:nist}. The steps of the derivation parallel those used in Appendix \ref{app:e1}, except that here we use multiset combinatorial coefficients  to expand $(1 - (\nu-1) u)^{-(2g+\ell-1)}$ in line 6.
\medskip

The contour integral derivation of hypergeometric counts for $e_g$, $g \ge 2$, follows a similar line and makes use of equation \eqref{eq:eg_parfrac}, written as
\[
e_g(z_0)=C^{(g)}+\sum_{\ell=0}^{3g-3}\dfrac{b_\ell^{(0)}(g,\nu)}{(\nu-(\nu-1) z_0)^{2g+\ell-2}}.
\]
We note that the coefficients $b_\ell^{(0)}(g,\nu)$ depend on $\nu$, whereas $C^{(g)}$ does not. For clarity, in the calculations below, we write $b_\ell^{(0)}$ instead of $b_\ell^{(0)}(g,\nu)$. Since the counts for $j=0$ are zero, we assume $j\geq 1$.

\begin{eqnarray*}
{\mathcal N}_{2\nu,e}(g,j) &=& \dfrac{j!\, c_\nu^j}{2\pi i}\oint \dfrac{e_g(\eta)}{\eta^{j+1}} d\eta\\
&=& j!\, c_\nu^j \left\{\frac{C^{(g)}}{2\pi i} \oint \dfrac{d\eta}{\eta^{j+1}}  + \sum_{\ell = 0}^{3g-3} b_\ell^{(0)}\frac1{2\pi i} \oint \frac{z^{\nu j -1} }{(z-1)^{j+1} (\nu - (\nu-1) z)^{2g+\ell-3}}  dz \right\}\\
&=& j!\, c_\nu^j \sum_{\ell = 0}^{3g-3} b_\ell^{(0)} \frac1{2\pi i} \oint \frac{(1+u)^{\nu j -1} }{u^{j+1} (1 - (\nu-1) u)^{2g+\ell-3}}  du\\
&=& j!\, c_\nu^j   \sum_{\ell = 0}^{3g-3} b_\ell^{(0)}\frac1{2\pi i} \oint \frac{\sum_{n=0}^\infty\left({2g-3+\ell \choose n}\right) ((\nu-1)u)^n \sum_{m=0}^{\nu j-1} { \nu j - 1 \choose m} u^m}{u^{j+1} }  du \\
&=& j!\, c_\nu^j   \sum_{\ell = 0}^{3g-3} b_\ell^{(0)}
\sum_{n=0}^\infty (\nu-1)^n \sum_{m=0}^{\nu j -1}  \left({2g-3+\ell \choose n}\right)   { \nu j - 1\choose m} \frac1{2\pi i} \oint u^{m +n - j -1}  du \\
&=& j!\, c_\nu^j   \sum_{\ell = 0}^{3g-3} b_\ell^{(0)}
\sum_{m=0}^{ j} (\nu-1)^{j-m} {(2g-4) + (\ell+j) -m \choose j-m}   { \nu j - 1\choose m}\\
&=& j!\, c_\nu^j (\nu - 1)^j  \sum_{\ell = 0}^{3g-3} b_\ell^{(0)}
{(2g-4) + (\ell+j)  \choose j}  \sum_{m=0}^{ \infty} \frac{ (1-\nu j)^{\overline{m}}  ( -j)^{\overline{m}}}{((4-2g) -(\ell+j) )^{\overline{m}}}  \frac{(1-\nu)^{-m}}{m!}\\
&=& j!\, c_\nu^j (\nu - 1)^j  \sum_{\ell = 0}^{3g-3} b_\ell^{(0)}
{(2g-4) + (\ell+j)  \choose j}
\ {_2}F_1
\left( \genfrac{}{}{0pt}{}
{-j,\ 1- \nu j}{4-2g-(\ell+j)}; \frac{1}{1-\nu} \right)
\end{eqnarray*}
The hypergeometric functions appearing in  ${\mathcal N}_{2\nu,z}(g,j)$ and ${\mathcal N}_{2\nu,e}(g,j)$ are well defined even though their denominator involves a negative parameter. This is because the corresponding terms in the sum already vanish through the rising factorial of $-j$. This convention is in agreement with the definition of the Gauss hypergeometric function given in \cite{bib:nist}, Equation 15.2.5, and is also used by Maple \cite{bib:maple} when evaluating these functions.

\section{Dynamics of the vectors $V^{(j)}$ under the linear transformations $M^{(n)}$} \label{app:4-valent_dynamics}

This appendix uses notation introduced in Section \ref{sec:rec_relations}. Because the partial fraction expansion of $R^{(j)}(z_0)$ involves a finite number of terms, any semi-infinite vector $V^{(j)}$ has a finite number of non-zero entries. We may thus ask how the number of such non-zero entries evolves with $j$. 

\begin{figure}[hbtp]
    \centering
    \includegraphics[width=\linewidth]{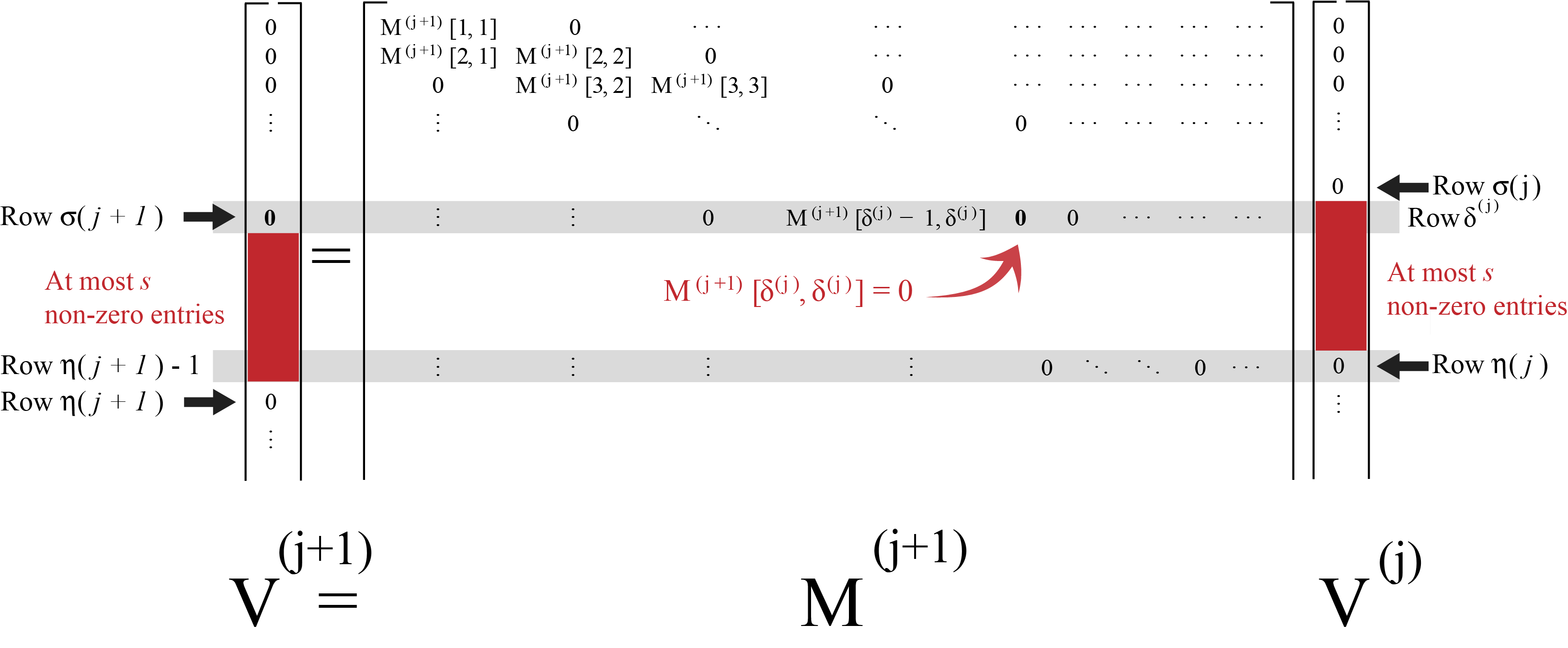}
    \caption{This figure illustrates that if $V^{(j)}$ has a band of length $s$ of possibly non-zero entries, from row $\sigma(j)+1$ to row $\eta(j)-1$, and if $M^{(j+1)}[\delta^{(j)},\delta^{(j)}]=0$ (in bold), where $\delta^{(j)}=\sigma(j)+1$, then the product $V^{(j+1)} = M^{(j+1)} V^{(j)}$ has a zero in slot $\delta^{(j)}$ (also in bold) followed by at most $s$ non-zero entries, from row $\sigma(j+1)+1$ to row $\eta(j+1)-1$. In each vector, the band of length $s = \eta(j)-1-\sigma(j) = \eta(j+1)-1-\sigma(j+1)$ is represented by a red rectangle. The pattern repeats for $V^{(j+2)}$ as long as $\delta(j+1)=\delta(j)+1$.}
    \label{fig:MV}
\end{figure}

\subsection{Number of non-zero entries of $V^{(j)}$}
For each vector $V^{(j)}$, we call $\sigma(j)$ the index marking the end of some initial band of zeros, and $\eta(j)$ the index marking the beginning of the semi-infinite string of zeros in $V^{(j)}$. In the 4-valent case, that is when $\nu = 2$, the following analysis proves that $s(j) := \eta(j) - 1 - \sigma(j)$, the generic number of non-zero entries of $V^{(j)}$, is constant for $j$ large enough. The reader is referred to Figure \ref{fig:MV} for the notation related to Theorem \ref{thm:band}.

\begin{theorem} \label{thm:band}
Assume $\delta^{(j+1)} = \delta^{(j)} + 1$ and let $j$ be large enough such that $\sigma(j) = \delta^{(j)} - 1 \ge 0$. Let $\eta(j)$ be such that $\eta(j) > \sigma(j)+1 = \delta^{(j)}$ and
\[
V^{(j)}[i] = 0 \text{ for } i \ge \eta(j).
\]
Then, one can choose $\eta(j)$ such that $\eta(j+1) = \eta(j) + 1$. If, in addition,
\[
V^{(j)}[i] = 0 \text{ for } i \le \sigma(j),
\]
then the length of the band of possibly non-zero entries of $V^{(j)}$, $s = \eta(j) - 1 - \sigma(j)$, is independent of $j$.
\end{theorem} 
\begin{proof}
For $i \ge 1$ and $j \ge j_G$ such that $\sigma(j) = \delta^{(j)} - 1 \ge 0$, suppose that $V^{(j)}[i] = 0 \text{ for } i \le \sigma(j)$. We calculate $V^{(j+1)}[i]$, using the structure of the matrix $M^{(j+1)}.$ Since
\[
V^{(j+1)}[i] = c_\nu \left(M^{(j+1)}[i,i-1] \cdot  V^{(j)}[i-1] + M^{(j+1)}[i,i] \cdot V^{(j)}[i]\right),
\]
we can establish the following.
\begin{itemize}
\item $\displaystyle V^{(j+1)}[i] = 0$ for $i \le \sigma(j)$ since both $V^{(j)}[i-1]$ and $V^{(j)}[i]$ are zero. 
\item $\displaystyle V^{(j+1)}[\sigma(j)+1] = -c_\nu \, M^{(j+1)}[\sigma(j)+1,\sigma(j)+1] \ V^{(j)}[\sigma(j)+1] = 0$ since $\sigma(j)+1 = \delta^{(j)}$ and $\displaystyle M^{(j+1)}[\delta^{(j)},\delta^{(j)}]=0$.
\end{itemize}
Therefore, $V^{(j+1)}[i] = 0 \text{ for } i \le \sigma(j)+1 =  \delta^{(j)} + 2 = \delta^{(j+1)} + 1 = \sigma(j+1)$. Note that crucial to this step is the assumption that $\delta^{(j+1)} = \delta^{(j)} + 1$. Similarly,
\begin{itemize}
\item $V^{(j+1)}[i] = 0$ for $i \ge \eta(j)+1$ since both $V^{(j)}[i-1]$ and $V^{(j)}[i]$ are zero. We can thus define $\eta(j+1) = \eta(j) + 1$.
\end{itemize}
Therefore, $V^{(j+1)}[i] = 0 \text{ for } i \ge \eta(j+1)$. 
Now, consider the length $s(j) = (\eta(j) - 1) - \sigma(j)$ of the vector $X^{(j)} = \left[V^{(j)}[\sigma(j)+1], \cdots, V^{(j)}[\eta(j)-1]\right]$ consisting of the band of possibly non-zero entries of $V^{(j)}$. We have
\[
s(j+1) = \eta(j+1)-1-\sigma(j+1) = \big(\eta(j) + 1\big) - 1 - \big(\sigma(j) + 1\big) = \eta(j)-1-\sigma(j) = s(j)
\]
and is therefore independent of $j$.
\end{proof}
This Theorem is only useful if $\delta^{(j+1)} = \delta^{(j)} + 1$. This is the case when $\nu = 2$, as shown below.

\begin{lemma} \label{lm:j0}
When $\nu = 2$, 
\[
\delta^{(j+1)} = \delta^{(j)} + 1.
\]
In addition, 
\[
\delta^{(j)} \ge 1 \text{ for } j \ge j_0 = \beta - 1.
\]
\end{lemma}
\begin{proof}
Recall that $\delta^{(j)}$ is the unique index such that $M^{(j+1)}\Big[\delta^{(j)},\delta^{(j)}\Big] = 0$, where $M^{(j)}[k,k]$ is given in Theorem \ref{thm:M}. Since $\delta^{(j)} = j (\nu-1)-\beta+2$, this index is linear in $j$ with slope equal to $1$ when $\nu = 2$, so that $\delta^{(j+1)}= \delta^{(j)} + 1$. The last statement of the lemma is obtained by solving $\delta^{(j_0)} = 1$ for $j_0$ and noting that the $\delta^{(j)}$ are increasing functions of $j$.
\end{proof}

Since the conditions of Theorem \ref{thm:band} apply when $j = j_0$, we have the following result.
\begin{lemma} \label{lm:N_val}
When $\nu = 2$, we can choose the length $s$ of the vector
\[
X^{(j)} = \left[V^{(j)}[h(j)+1], \cdots, V^{(j)}[k(j)-1]\right]
\]
to be such that
\[
s = \alpha + \beta.
\]
\end{lemma}
\begin{proof}
Equation \eqref{eq:df} shows that when $j=j_0$, the number of terms in the partial fraction expansion of $R^{(j)}(z_0)$ is no larger than $\alpha + j_0 + 1$, which can then be used to define $s$.
\end{proof}

\begin{rem}
Clearly, since the value of $\eta(j)$ can be arbitrarily increased, any value of $s$ larger than that set in Lemma \ref{lm:N_val} is also valid. However, we know that $\alpha + \beta = 5 g - 1$ (resp. $\alpha + \beta = 5 g - 5$) is the exact number of non-zero coefficients in the partial fraction expansion of $R^{(j)}(z_0)$ when $G=z_g$ (resp. $G=e_g$) \cite{bib:er,bib:er14}.
\end{rem}

To summarize, we have shown that when $\nu = 2$ and $j \ge j_0 \ge j_G$, the vector $V^{(j)}$ has the form
\[
V^{(j)} = [0, \cdots, 0, x_1^{(j)}, x_2^{(j)}, \cdots, x_s^{(j)}, 0, \cdots]
\]
where $x_k^{(j)} = V^{(j)}[\sigma(j)+k]$ and we defined
$X^{(j)} = [x_1^{(j)}, \cdots, x_s^{(j)}]$. 
Recall that $j_G$ is the lowest value of $j$ for which $R^{(j)}(z_0)$ has a partial fraction expansion that satisfies the form assumed to derive the recurrence relations of Theorem \ref{thm:M}. 

\begin{rem}
\label{rem:jG_less_j0}
If $j_G < j_0$, then the recurrence relations still apply and the above results may be extended to values of $j$ such that $j_0 > j \ge j_G$ provided the partial fraction expansion of $R^{(j)}$ has no more than $s$ terms. In such a case, it is still possible to define $X^{(j)}$ as above, by padding the vector of initial coefficients with zeros on the left, to make it exactly of size $s$. From Equation \eqref{eq:df}, we see that the number of terms in the partial fraction expansion of $R^{(j)}$ is equal to $\alpha + j + 1$. Setting $\alpha + j + 1 \le s$ gives $j \le \beta - 1 = j_0.$ Therefore, the structure of $V^{(j)}$ and the definition of $X^{(j)}$ are valid for all $j \ge j_G$ as long as $j_G \le j_0$.
\end{rem}

\begin{rem} 
\label{rem:jG_greater_j0}
If $j_G > j_0$, the condition 
\[
V^{(j)}[i] = 0 \text{ for } i \le \sigma(j),
\]
specified in Theorem \ref{thm:band} is never satisfied. Indeed, for $q_0^{(j+1)}$ to be zero when $q_0^{(j)} \ne 0$, $d^j G / d t^j$ should be in the form prescribed by Equation \eqref{eq:df} before or at the same time as $\delta^{(j)}$ becomes equal to one. Since $\delta^{(j)}=1$ when $j = j_0$, one needs $j_0 \ge j_G$.
\end{rem}

\subsection{Reduction of the dynamics to a vector of fixed length}
Since when $\nu = 2$, the length $s$ is independent of $j$ for $j \ge j_G$, the evolution of $V^{(j)}$ under the linear dynamics corresponding to successive applications of $M^{(n)}$ can be reduced to a description of how the vector $X^{(j)}$ maps into $X^{(j+1)}$. 

\begin{theorem} \label{thm:A_matrix}
The vector $X^{(j+1)}$ is obtained from $X^{(j)}$ according to the linear transformation
$X^{(j+1)} = c_2 \, A^{(j)} X^{(j)},$ where $X^{(j)} = [x_1^{(j)}, \cdots, x_s^{(j)}]$ and the matrix $A^{(j)}$ is the $s \times s$ truncation of the semi-infinite matrix
\begin{equation}
    \label{eq:Matrix_A}
A^{(j,\infty)} = \begin{bmatrix}
2 (2j+1) & -1 & 0 & \cdots\\
0 & 2(2j+2) & -2 & 0 & \cdots \\
\cdots & 0 & 2(2j+3) & -3 & 0 & \cdots \\
\vdots & \qquad \ddots & \qquad \ddots & \qquad \ddots & \qquad \ddots & \cdots\\
\cdots & \cdots & \qquad 0 & \qquad 2(2j+k) & \qquad -k & \qquad 0 & \cdots \\
\vdots & \vdots & \vdots & \quad \ddots & \quad \ddots & \quad \ddots & \quad \ddots
\end{bmatrix}.
\end{equation}
\end{theorem}
\begin{proof}
This is obtained by direct calculation, noting that $x_k^{(j)} = V^{(j)}[\sigma(j)+k].$ Specifically,
\begin{align*}
\frac{1}{c_2} x_k^{(j+1)} & = \frac{1}{c_2} V^{(j+1)}[\sigma(j+1)+k] = \frac{1}{c_2} V^{(j+1)}[\sigma(j)+1+k]\\
& = M^{(j+1)}[\sigma(j)+1+k,\sigma(j)+k] \cdot V^{(j)}[\sigma(j)+k] \\
& \qquad + M^{(j+1)}[\sigma(j)+1+k,\sigma(j)+1+k] \cdot V^{(j)}[\sigma(j)+k+1]\\
& = M^{(j+1)}[\sigma(j)+1+k,\sigma(j)+k] \cdot x_k^{(j)} + M^{(j+1)}[\sigma(j)+1+k,\sigma(j)+1+k] \cdot x_{k+1}^{(j)}.
\end{align*}
We can calculate
\begin{align*}
M^{(j+1)}[\sigma(j)+1+k,\sigma(j)+k] & = M^{(j+1)}[j-\beta+2+k,j-\beta+1+k] \\
& = 2(\beta+j-\beta+2+k+j+1-3) = 2(2j+k)
\end{align*}
and 
\begin{align*}
M^{(j+1)}[\sigma(j)+1+k,\sigma(j)+1+k]
& = M^{(j+1)}[j-\beta+2+k,j-\beta+2+k] \\ & = j+1-(\beta+j-\beta+2+k-1) = - k.
\end{align*}
Therefore,
\[
\frac{1}{c_2}\, x_k^{(j+1)} =  2(2j+k)\ x_k^{(j)} - k\ x_{k+1}^{(j)}, \qquad 1 \le k \le s
\]
with the understanding that $x_{s+1}^{(j)} = 0$. This proves the statement.
\end{proof}

\section{Map Counts as Vectorial Contractions}
\label{app:map_cts_dot_product}
We now return to the definitions of ${\mathcal N}_{4,z}(g,j)$ and ${\mathcal N}_{4,e}(g,j)$ in Equations  \eqref{eq:map_counts_z} and \eqref{eq:map_counts_e}. In terms of $G$, we are interested in calculating
\[
{\mathcal N}_4(j) = \left. (-1)^j \, \dfrac{d^j G}{d t^j}\right\vert_{t=0} = G^{(j)}(1) = R^{(j)}(1)
\]
since $z_0 = 1$ when $t = 0$.
Because $V^{(j)}$ contains the coefficients of the partial fraction expansion of $R^{(j)}(z_0)$ in powers of $(2-z_0)$ (which is equal to $1$ when $z_0 = 1$), the sum of the entries of $V^{(j)}$ is exactly equal to $R^{(j)}(1)$. Consequently, since the section of $V^{(j)}$ with non-zero entries is equal to the vector $X^{(j)}$, we have
\begin{equation}
\label{eq:counts}
{\mathcal N}_4(j) = \sum_{i=1}^{s} X^{(j)}[i].
\end{equation}
The problem of counting maps has therefore been reduced to finding the sum of entries of a vector resulting from successive applications of the matrices $A^{(j)}$. Specifically,
\[
X^{(j)} = c_2^{j-j_G} \left(\overleftarrow{\prod_{k=j_G}^{j-1}} 
A^{(k)} \right) X^{(j_G)},
\]
where $j_G$ is the lowest value of $j$ for which $R^{(j)}(z_0)$ has a partial fraction expansion of the form given in \eqref{eq:df}. This allows us to restate Equation \eqref{eq:counts} in terms of the following lemma.
\begin{lemma} \label{lm:counts}
For $\nu = 2$, let ${\mathcal R}^{(j)}$ be the row vector obtained by summing the columns of the matrix
\[
{\mathcal M}^{(j)} = \overleftarrow{\prod_{k=j_G}^{j-1}} 
A^{(k)}.
\]
The number of 4-valent, 2-legged (resp. regular) maps with $j$ vertices that can be embedded in a surface of genus $g$ is given by the contraction
\[
{\mathcal N}_4(j) = c_2^{j-j_G}\ {\mathcal R}^{(j)} \cdot X^{(j_G)}
\]
where $G = z_g$ (resp. $G=e_g$). For reference, the $s \times s$ matrix $A^{(k)}$ is defined in Theorem \ref{thm:A_matrix}, $s = \alpha + \beta$ was calculated in Lemma \ref{lm:N_val}, and the values of $j_G$, $\alpha$, and $\beta$ for $z_g$ and $e_g$ are given in Equation \eqref{eq:param}.
\end{lemma}


The next theorem exploits the remarkable structure of the $A^{(j,\infty)}$ matrices to describe the structure of the row vectors ${\mathcal R}^{(j)}$.

\begin{theorem} \label{thm:row_vector}
The $n^{\rm{th}}$ entry of the row vector ${\mathcal R}^{(j)}$ satisfies
\[
{\mathcal R}^{(j)}[n] = \frac{1}{2^{n-1}} \sum_{k=0}^{n-1} \binom{n-1}{k} d_{k+1}^{(j)}, \qquad d_m^{(j)} = \prod_{\ell=j_G}^{j-1}
2 (2 \ell + m).
\]
\end{theorem}
\begin{proof} The proof is based of 2 lemmas.
\begin{lemma} \label{lem:commute_A}
The $s \times s$ matrices $A^{(j)}$ and $A^{(k)}$ commute and therefore share a basis of eigenvectors.
\end{lemma}
\begin{proof}
From 
Equation \eqref{eq:Matrix_A}, one may write that $A^{(j)} = D_j + N$ where $D_j$ is a diagonal matrix and  $N$ is a matrix with non-zero entries $N[i, i+1] = -i$. From this it is immediate that the commutator
\begin{eqnarray*}
\left[A^{(j)}, A^{(k)}
\right] &=& \left[N, D_k - D_j \right]\\
&=& 4(k-j) \left[ N, \mathbb{I}\right] = 0.
\end{eqnarray*}
\end{proof}

\begin{lemma} \label{lem:a_matrix}
Each $s \times s$ matrix $A^{(\ell)}$ may be written as
\[
A^{(\ell)} = S \cdot D^{(\ell)} \cdot S^{-1}
\]
where the diagonal entries of $D^{(\ell)}$ are equal to $2(2 \ell + k)$, for $k=1, \cdots, s$. The matrix $S$ (whose columns are the common eigenvectors of the $A^{(\ell)}$, for all $\ell$) may be chosen to be an upper unipotent matrix with entries
\begin{equation} \label{eq:s_entries}
S[k-m,k] = (-1)^m \left(\frac12\right)^m 
{k-1 \choose m}, \qquad 0 \le m < k, \quad 1 \le k \le s.
\end{equation}
The entries of the inverse, $S^{-1}$ are given by
\begin{equation} \label{eq:s_inv_entries}
S^{-1}[k-m,k] = \Big\vert S[k-m,k] \Big\vert = \left(\frac12\right)^m
{k-1 \choose m}, \qquad 0 \le m < k, \quad 1 \le k \le s.
\end{equation}
\end{lemma}
\begin{proof}
Since, by Lemma \ref{lem:commute_A}, the eigenvectors of the $A^{(\ell)}$ are common, to find $S$ it suffices to solve
\[
A^{(0)} = S \cdot D^{(0)} \cdot S^{-1}.
\]
For the $k^{th}$ eigenvector, $S[\cdot,k]$, one must have 
\begin{equation*}
    A^{(0)}S[\cdot,k] = 2k S[\cdot,k]
\end{equation*}
which from the 2-step structure of $A^{(0)}$ amounts to the system (for $1 \le m <k$)
\begin{eqnarray} \label{2step1}
2k S[k,k] &=& 2k S[k,k]\\ \label{2step2}
2(k-m) S[k-m,k] - (k-m)  S[k-m+1,k]
&=& 2k S[k-m,k]. 
\end{eqnarray}
From \eqref{2step1} it follows that one may take $S[k,k]=1$ so that $S$ will be unipotent. Equation \eqref{2step2} may be rewritten as the fundamental relation
\begin{equation} \label{2StepFun}
    S[k-m,k] = -\dfrac{(k-m)}{2m} S[k-(m-1),k], \qquad 1 \le m < k.
\end{equation}
The first statement of the lemma may now be proved by induction on $m$, with the base step (for $m = 1$) stemming from the choice $S[k,k]=1$:
\[
S[k-1,k]=-\dfrac{(k-1)}{2} S[k,k] = (-1)^1 \left(\frac12\right)^1 
{k-1 \choose 1}.
\]
For the induction step suppose one knows that the lemma holds for $m$. Applying this to \eqref{2StepFun} yields
\begin{eqnarray*}
    S[k-(m+1),k] &=& -\dfrac{(k-(m+1))}{2(m+1)} S[k-m,k]\\
    &=& -\dfrac{(k-(m+1))}{2(m+1)} (-1)^m \left(\frac12\right)^m 
{k-1 \choose m}\\
&=& (-1)^{m+1} \left(\frac12\right)^{m+1} \dfrac{(k-m-1)}{m+1} \dfrac{(k-1)!}{m! (k - 1 - m)!}\\
&=& (-1)^{m+1} \left(\frac12\right)^{m+1}  \dfrac{(k-1)!}{(m+1)! (k - 1 - (m+1))!}\\
&=& (-1)^{m+1} \left(\frac12\right)^{m+1}  {k-1 \choose m+1}.
\end{eqnarray*}
Finally, the second statement of the lemma follows from the first but is a bit more involved. It is proved by induction on the size $s$ of $S$. For this reason, it will be convenient to think of $S$ as the $s \times s$ truncation of a semi-infinite upper unipotent matrix whose columns are defined by Equation \eqref{eq:s_entries}, for $k \ge 1$. For $s=2$, one has from the prescriptions given in the statements of the lemma  that
\begin{equation*}
    S\cdot S^{-1} = \left(\begin{array}{cc}
        1 & -1/2 \\
        0 & 1
    \end{array}\right)
    \left(\begin{array}{cc}
        1 & 1/2 \\
        0 & 1
    \end{array}\right) = \left(\begin{array}{cc}
        1 & 0 \\
        0 & 1
    \end{array}\right).
\end{equation*}
Now assume that the claim about the entries of $S^{-1}$ holds for size $s-1$ and consider the case where $S$ has size $s$. By induction the inverse of the principal $(s-1)\times (s-1)$ sub-matrix of $S$ is the $(s-1)\times (s-1)$ matrix whose entries are $\big|S[m,n]\big|$ for $0 < m,n < s$. So to complete the induction step one just needs to show that contracting the first $s-1$ rows of $S$ with the last column of $S^{-1}$ (as specified in the lemma statement) always gives zero. The first statement of the lemma provides all the information needed to carry out these evaluations.  For $j = 0, \dots, s-2$ we want to show the vanishing of the following contraction:
\begin{eqnarray*}
&&\sum_{m=j}^{s-1} S[j+1,m+1]
\Big\vert S[m+1, s]\Big\vert\\
&=& \sum_{m=j}^{s-1}
\left( (-1)^{m-j} \dfrac{1}{2^{m-j}}{m \choose m-j}\right)
\left(\dfrac{1}{2^{s-m-1}}{s-1 \choose s-m-1}\right)\\
&=& \sum_{m=j}^{s-1} (-1)^{m-j} \dfrac{1}{2^{s-j-1}}
  {m \choose m-j}
 {s-1 \choose s-m-1}\\
 &=& \dfrac{1}{2^{s-j-1}}
 \sum_{m=j}^{s-1} (-1)^{m-j} 
  {m \choose m-j}
 {s-1 \choose s-m-1}\\
 &=& \dfrac{1}{2^{s-j-1}}
 \sum_{m=j}^{s-1} (-1)^{m-j} 
 {m \choose j}{s-1 \choose m}.
\end{eqnarray*}
The last sum may be rewritten as
\begin{equation*}
  \dfrac{1}{j!} \sum_{m=j}^{s-1} (-1)^{m-j}
 {s-1 \choose m}  m (m-1) \cdots (m-j+1).
\end{equation*}
A moment's reflection reveals that this sum may be re-expressed as
\begin{equation*}
  \dfrac{1}{j!} \dfrac{d^j}{dx^j}\sum_{m=0}^{s-1}
 {s-1 \choose m}  x^m \Big|_{x=-1}
 = \dfrac{1}{j!} \dfrac{d^j}{dx^j}
 (1 + x)^{s-1}\Big|_{x=-1}.
\end{equation*}
which clearly vanishes for $j=0,\dots,s-2$. This concludes the proof of Lemma \ref{lem:a_matrix}.
\end{proof}

Therefore,
\begin{align*}
{\mathcal M}^{(j)} & = \prod_{k=j_G}^{j-1} 
A^{(k)} = S \cdot \prod_{k=j_G}^{j-1} D^{(k)} \cdot S^{-1} \\
& = S \cdot {\rm diag} \left(\left[\prod_{k=j_G}^{j-1} 2 (2 k + 1), \prod_{k=j_G}^{j-1} 2 (2 k + 2), \cdots, \prod_{k=j_G}^{j-1} 2 (2 k + s),\right] \right)\cdot S^{-1}\\
& =  S \cdot {\rm diag}\left( \left[d_1^{(j)}, d_2^{(j)}, \cdots, d_s^{(j)}\right]\right) \cdot S^{-1}
\end{align*}

The final step consists in using the structure of the matrix $S$ to calculate the column sums of ${\mathcal M}^{(j)}$, which are the entries of the row vector ${\mathcal R}^{(j)}$. The
$n^{th}$ column vector of $\mathcal {M}^{(j)}$ is
\begin{eqnarray*}
{\mathcal M}^{(j)}[r,n] &=&
\sum_{\ell = r}^s S[r,\ell] d^{(j)}_\ell S^{-1}[\ell,n] \\
& = & \sum_{\ell = r}^s 
\left( (-1)^{\ell-r} \frac1{2^{\ell-r}} {\ell-1 \choose \ell - r}\right) d^{(j)}_\ell 
\left( \frac1{2^{n-\ell}} {n - 1 \choose n - \ell}\right)\\
&=& \frac1{2^{n-r}} \sum_{\ell = r}^s (-1)^{\ell-r} {\ell-1 \choose \ell - r} {n - 1 \choose n - \ell}
d^{(j)}_\ell
\end{eqnarray*}
It follows that the $n^{th}$ entry of ${\mathcal R}^{(j)}$ is
\begin{eqnarray*}
{\mathcal R}^{(j)}[n] &=& 
\sum_{r = 1}^n \frac1{2^{n-r}} \sum_{\ell = r}^s (-1)^{\ell-r} {\ell-1 \choose \ell - r} {n - 1 \choose n - \ell}
d^{(j)}_\ell \\
&=& \sum_{\ell = 1}^n d^{(j)}_\ell  {n-1 \choose n- \ell} \sum_{r = 1}^\ell\frac1{2^{n-r}}  (-1)^{\ell-r} {\ell-1 \choose \ell - r} \\
&=& \sum_{\ell = 1}^n \frac1{2^{n-\ell}} d^{(j)}_\ell  {n-1 \choose n- \ell} \sum_{r = 1}^\ell\frac1{2^{\ell-r}}  (-1)^{\ell-r} {\ell-1 \choose \ell - r} \\
&=& \sum_{\ell = 1}^n \frac1{2^{n-\ell}} d^{(j)}_\ell  {n-1 \choose n- \ell} \sum_{r = 0}^{\ell - 1} \frac1{2^{(\ell - 1) -r}}  (-1)^{(\ell - 1) -r} {\ell-1 \choose (\ell - 1) - r} \\
&=& \sum_{\ell = 1}^n \frac1{2^{n-\ell}} d^{(j)}_\ell  {n-1 \choose n- \ell} \sum_{r = 0}^{\ell - 1} \frac1{2^{r}}  (-1)^{r} {\ell-1 \choose r}\\
&=& \sum_{\ell = 1}^n \frac1{2^{n-\ell}} d^{(j)}_\ell  {n-1 \choose n- \ell} \left( 1 - x \right)^{\ell - 1}\Big|_{x = 1/2}\\
&=& \frac1{2^{n-1}} \sum_{\ell = 1}^n  d^{(j)}_\ell  {n-1 \choose n- \ell} \\
&=& \frac1{2^{n-1}} \sum_{\ell = 0}^{n-1}   {n-1 \choose \ell}
d^{(j)}_{\ell + 1}.
\end{eqnarray*}
This concludes the proof of Theorem \ref{thm:row_vector}.
\end{proof}

The results of this section are summarized in Theorem \ref{thm:gen_counts} of the main text.

\section{Initial vectors of coefficients $V_{z,g}^{(0)}$ for ${\mathcal N}_{4,z}(g,j)$} \label{app:zg}
Each initial vector $V_{z,g}^{(0)}$ below was obtained by finding the partial fraction expansion of $z_g(z_0)/z_0$ in powers of $2-z_0$, using the expressions provided in \cite{bib:elt22b}. The resulting set of non-zero coefficients was padded on the left by $j_0 - j_G = j_0 = \beta - 1 = 2 g - 1$ zeros, to create a vector of length $s = \alpha + \beta = 5 g - 1$.
\begin{align*}
V_{z,2}^{(0)}&=\left[0, 0, 0, -14, \dfrac{700}{9}, - \dfrac{1540}{9}, \dfrac{560}{3}, - \dfrac{910}{9}, \dfrac{196}{9}\right]^T\\
V_{z,3}^{(0)}&= \left[0, 0, 0, 0, 0, \dfrac{10796}{9}, -\dfrac{297416}{27}, \dfrac{1182748}{27}, -\dfrac{888160}{9}, \dfrac{3723580}{27}, - \dfrac{
3308648}{27}, \dfrac{608972}{9}, \right. \\
& \qquad \left. - \dfrac{573488}{27}, \dfrac{78400}
{27} \right]^T
\\
V_{z,4}^{(0)}&= \left[0, 0, 0, 0, 0, 0, 0, -\dfrac{18696694}{81}, \dfrac{26661644}{9}
, - \dfrac{1387088600}{81}, \dfrac{4772537840}{81}, - \dfrac{
10869779620}{81}, \right.\\
& \qquad \dfrac{17220205688}{81}, - \dfrac{19375553512}{81}
, \dfrac{15491885600}{81}, - \dfrac{8630131990}{81}, \dfrac{
1063795540}{27}, - \dfrac{705318544}{81}, \\
& \qquad \left. \dfrac{70598416}{81}\right
]^T
\\
V_{z,5}^{(0)}&= \left[0, 0, 0, 0, 0, 0, 0, 0, 0, \dfrac{709788436}{9}, - \dfrac{
105371101864}{81}, \dfrac{800494094596}{81}, - \dfrac{412819350944}{9
}, \right. \\
& \qquad \dfrac{11778382939400}{81}, - \dfrac{26995581694928}{81}, \dfrac{
5128257594152}{9}, - \dfrac{2215723834016}{3}, 729617439012, \\
& \qquad - \dfrac{
4921613081000}{9}, \dfrac{24801764653204}{81}, - \dfrac{
10063123626304}{81}, \dfrac{310842049504}{9}, - \frac{477090039776}{
81}, \\
& \left. \qquad \dfrac{37662732800}{81}\right]^T
 \\
V_{z,6}^{(0)}&= \left[0, 0, 0, 0, 0, 0, 0, 0, 0, 0, 0, - \dfrac{10225685162212}{243}, 
\dfrac{618460530794968}{729}, - \dfrac{5825738824972472}{729}, \right. \\
& \qquad \dfrac
{34075590428415200}{729}, - \dfrac{138689846516551540}{729}, \dfrac{
417032295219846488}{729}, \\
& \qquad - \dfrac{959737294012659344}{729}, \dfrac{
191919774489981824}{81}, - \dfrac{820677231516099100}{243}, \\
& \qquad \dfrac{
932026632424847960}{243}, - \dfrac{2530726899989856568}{729}, \dfrac{
1815557475898877984}{729}, \\
& \qquad - \dfrac{1019969072098037516}{729}, \dfrac
{439393492396822600}{729}, - \dfrac{140196204075971840}{729}, \\
& \qquad \left. \dfrac{
31224729704537216}{729}, - \dfrac{1444996352993536}{243}, \dfrac{
31389368571008}{81}\right]^T\\
V_{z,7}^{(0)}&= \left[0, 0, 0, 0, 0, 0, 0, 0, 0, 0, 0, 0, 0, \dfrac{70696890174658568}
{2187}, - \dfrac{1684587188506923920}{2187}, \right. \\
& \qquad \dfrac{
6315201936216281224}{729}, - \dfrac{133780627204397704256}{2187}, 
\dfrac{665498065031354080456}{2187}, \\
& \qquad - \dfrac{2480051038872504782384}{
2187}, \dfrac{7186439203612748889880}{2187}, - \dfrac{
16585967336330196209696}{2187}, \\
& \qquad \dfrac{30973939081698653980888}{2187}
, - \dfrac{15759088207491625771216}{729}, \dfrac{
59316651347289200585864}{2187}, \\
& \qquad - \dfrac{61294373220356485664000}{2187
}, \dfrac{52084247212114918540376}{2187}, - \dfrac{
36202722253041172445008}{2187}, \\
& \qquad \dfrac{20386035169723888009928}{2187}
, - \dfrac{9158234740399344242848}{2187}, \dfrac{
3205807819986945943040}{2187}, \\
& \qquad - \dfrac{842821687710905844352}{2187}, \dfrac{156579118841594853376}{2187}, - \dfrac{6110092774214438912}{
729},\\
& \left. \qquad \dfrac{1017070902906060800}{2187} \right]^T
\end{align*}

\section{Initial vectors of coefficients $V_{e,g}^{(1)}$ for ${\mathcal N}_{4,e}(g,j)$} \label{app:eg}
Each initial vector $V_{e,g}^{(1)}$ below was obtained by finding the partial fraction expansion of $\dfrac{-1}{z_0^3} \dfrac{d e_g}{d t}$ in powers of $2-z_0$, using the expressions provided in \cite{bib:elt22b}. The resulting set of non-zero coefficients was padded on the left by $j_0 - j_G = j_0 - 1 = \beta - 2 = 2 g - 3$ zeros, to create a vector of length $s = \alpha + \beta = 5 g - 5$.
\begin{align*}
V_{e,2}^{(1)}&=\left[0, - \dfrac{13}{3}, 18, -23, \dfrac{28}{3}\right]^T\\
V_{e,3}^{(1)}&= \left[0, 0, 0, \dfrac{5482}{27}, - \dfrac{14620}{9}, \dfrac{15580}{3}, - \dfrac{231920}{27}, \dfrac{70270}{9}, - 3716, \dfrac{19600}{27}
\right]^T\\
V_{e,4}^{(1)}&= \left[0, 0, 0, 0, 0, - \dfrac{724697}{27}, 315602, - \dfrac{43334935}{
27}, 4659228, - \dfrac{76879901}{9}, \dfrac{277757942}{27}, - \frac{
73403603}{9}, \right. \\
& \qquad \left. \dfrac{37007360}{9}, - \dfrac{32343136}{27}, \dfrac{
4152848}{27}\right]^T
\\
V_{e,5}^{(1)}&= \left[0, 0, 0, 0, 0, 0, 0, \dfrac{566095186}{81}, -108169204, \dfrac{
60942443960}{81}, - \dfrac{28077019600}{9}, 8609216140, \right.\\
& \qquad - \dfrac{
450631849304}{27}, \dfrac{630238738696}{27}, -23759139360, \dfrac{
472123187750}{27}, - \dfrac{735685914340}{81}, \\
& \qquad \left. 3162718496, - \dfrac{53730218864}{81}, \dfrac{1711942400}{27}\right]^T
\\
V_{e,6}^{(1)}&= \left[0, 0, 0, 0, 0, 0, 0, 0, 0, - \dfrac{732368409218}{243}, \dfrac{
4681900125716}{81}, - \dfrac{123811531236922}{243}, \right.\\
& \qquad \dfrac{
2000274509112824}{729}, - \dfrac{273020590837748}{27}, \dfrac{
6572679622235672}{243}, - \dfrac{39591433575380588}{729}, \\
& \qquad \dfrac{
750851640690272}{9}, - \dfrac{2671231905442474}{27}, \dfrac{
2447012252991196}{27}, - \dfrac{5155389905557430}{81}, \\
& \qquad \dfrac{
8181175655761048}{243}, - \dfrac{9470356139656016}{729}, \dfrac{
93238615341712}{27}, - \dfrac{137458396716032}{243}, \\
& \left. \qquad \dfrac{
31389368571008}{729}\right]^T\\
V_{e,7}^{(1)}&= \left[0, 0, 0, 0, 0, 0, 0, 0, 0, 0, 0, \dfrac{1416512777883484}{729}
, - \dfrac{1199849052401896}{27}, \dfrac{345596095181486968}{729}, \right.\\
& \qquad - 
\dfrac{762707995529261600}{243}, \dfrac{10547915810649478660}{729}, -
\dfrac{35986409557097447672}{729}, \\
& \qquad \dfrac{10471595757825979504}{81}
, - \dfrac{193855395613589567104}{729}, \dfrac{317685851865765912100}{
729}, \\
& \qquad - \dfrac{418100684741264044520}{729}, \dfrac{147719729846848623464}{243}, - \dfrac{377558000806739867296}{729}, \\
& \qquad \dfrac{85576745657229258068}{243}, - \dfrac{137484666405365276200}{729}, \dfrac{56726601141030011360}{729}, \\
& \qquad - \dfrac{644643067895245952}{27}, \dfrac{3740709382598685952}{729}, - \dfrac{502622738156253568}{729}, \\
& \qquad \left. \dfrac{31783465715814400}{729}\right]^T
\end{align*}


\end{document}